\documentclass[11pt,centertags]{amsart}
\usepackage{amscd}
\usepackage{easybmat}
\usepackage{mathrsfs}
\usepackage{amsfonts}
\usepackage{color}
\usepackage{pifont}
\usepackage{upgreek}
\usepackage{bm}
\usepackage{hyperref}
\usepackage{shorttoc}
\usepackage{amsmath,amstext,amsthm,a4,amssymb,amscd}
\usepackage[mathscr]{eucal}
\usepackage{mathrsfs}
\usepackage{epsf}
\numberwithin{equation}{section}

\def\e{\epsilon}
\def\p{\partial}
\def\o{\overline}
\def\b{\bar}
\def\mb{\mathbb}
\def\mc{\mathcal}
\def\n{\nabla}

\newtheorem{thm}{Theorem}[section]
\newtheorem{lemma}[thm]{Lemma}
\newtheorem{prop}[thm]{Proposition}
\newtheorem{cor}[thm]{Corollary}
\theoremstyle{definition}
\newtheorem{rem}[thm]{Remark}

\theoremstyle{definition}
\newtheorem{defn}[thm]{Definition}
\newcommand{\comment}[1]{}
\usepackage{fancyhdr}

\begin{document}

\title{Positivity preserving along a flow over projective bundle}

\author{Xueyuan Wan}

\address{Mathematical Science Research Center, Chongqing University of Technology, Chongqing 400054, China}
\email{xwan@cqut.edu.cn}

\begin{abstract}
In this paper, we introduce a flow over the projective bundle $p:P(E^*)\to M$, which is a natural generalization of both Hermitian-Yang-Mills flow and  K\"ahler-Ricci flow. We prove that the semipositivity of curvature of the hyperplane line bundle $\mc{O}_{P(E^*)}(1)$ is preserved along this flow under the null eigenvector assumption. As applications, we prove that the semipositivity is preserved along the  flow if the base manifold $M$ is a curve, which implies that the Griffiths semipositivity is preserved along the Hermitian-Yang-Mills flow over a curve. And we also reprove that the nonnegativity of holomorphic bisectional curvature is preserved under K\"ahler-Ricci flow.
 \end{abstract}
\maketitle
\tableofcontents

\section*{Introduction}

In the celebrated paper \cite{Siu}, Siu and Yau presented a differential geometric proof of the
famous Frankel conjecture in K\"ahler geometry, which states that a compact K\"ahler manifold $M^n$ with positive holomorphic bisectional curvature must be biholomorphic to a complex projective space $\mb{P}^n$. For a compact K\"ahler manifold with nonnegative holomorphic bisectional curvature, Mok \cite{Mok} proved a generalized Frankel conjecture and obtained a uniformization theorem, and H. Gu \cite{Gu} gave a new proof of Mok's uniformization theorem. They both used the K\"ahler-Ricci flow and considered the variation of holomorphic bisectional curvature along this flow. Especially, the nonnegativity of holomorphic bisectional curvature is preserved under K\"ahler-Ricci flow \cite[Proposition 1.1]{Mok}.
    Later on, there are many references about studying and generalizing the Frankel conjecture by using K\"ahler-Ricci flow, include \cite{Chen, Chen1, Gu1, Phong}.

In 1979, Mori \cite{Mori} proved the famous Hartshorne's conjecture. In case the ground field is $\mb{C}$, a compact complex manifold $M$ was proved to biholomorphic to $\mb{P}^n$ if its tangent bundle is ample, which implies the Frankel conjecture. A holomorphic vector bundle $E$ is ample in the sense of Hartshorne if and only if the hyperplane line bundle $\mc{O}_{P(E^*)}(1)$ is a positive line bundle over $P(E^*)$ (see \cite[Proposition 3.2]{Har}), i.e., there is a positive curvature metric on $\mc{O}_{P(E^*)}(1)$. If $(E, h)$ is a Hermitian vector bundle with Griffiths positive curvature (see Definition \ref{defn2}), then $E$ is an ample vector bundle. In \cite{Gri}, Griffiths conjectured its converse also holds, namely $E$ can admit a Hermitian metric with Griffiths positive curvature if $E$ is ample. Both Mori's theorem and Griffiths conjecture would be proved if one can deform the given positive curvature metric to another ``better'' metric with positive curvature, for example, the K\"ahler metric with positive holomorphic bisectional curvature for Mori's theorem, and Hermitian metric with Girffiths positive curvature for Griffiths conjecture. Naturally, one wants to define a certain flow over the projective bundle $P(E^*)$ such that the positivity of the curvature  of $\mc{O}_{P(E^*)}(1)$ is preserved under this flow. This is also the motivation for the author to study the positivity preserving along the flow (\ref{flow0}) over projective bundle. 

Let $M$ be a compact complex manifold of dimension $n$, and $\pi: E\to M$ be a holomorphic vector bundle of rank $r$ over $M$. Let $E^*$ denote the dual bundle of $E$, and $p: P(E^*):=(E^*)^o/\mb{C}^*\to M$ denote the projective bundle, where $(E^*)^o$ denotes the set of all the nonzero elements in $E^*$. For any strongly pseudoconvex complex Finsler metric $G$ on $E^*$ (see Definition \ref{defn1}), there exists the following canonical decomposition (see Remark \ref{rem1} (4))
\begin{align*}
\sqrt{-1}\p\b{\p}\log G=-\Psi+\omega_{FS}(G),	
\end{align*}
which is a curvature form of $\mc{O}_{P(E^*)}(1)$ and represents the first Chern class $2\pi c_1(\mc{O}_{P(E^*)}(1))$.
Here $\Psi$ is called the Kobayashi curvature (see \cite[Definition 1.2]{FLW}), which is given in (\ref{Psi}), and $\omega_{FS}(G)$ is a positive $(1,1)$-form along each fiber of $p:P(E^*)\to M$, which is defined by (\ref{vertical form}). According to this decomposition, Kobayashi \cite{Ko1} gave a characterization of ample vector bundle, i.e., $E$ is ample if and only if there exists a strongly pseudoconvex complex Finsler metric on $E^*$ such that $\Psi<0$. 

Let $\omega(G)=p^*\omega$, where 
$
\omega=\sqrt{-1}g_{\alpha\b{\beta}}(G)dz^{\alpha}\wedge d\b{z}^{\beta} 
$
is a K\"ahler metric on $M$ depending smoothly on a Finsler metric $G$. 
One can define a Hermitian metric on $P(E^*)$ by
\begin{align*}
\Omega:=\omega(G)+\omega_{FS}(G),	
\end{align*}
Let $G_0$ be a strongly pseudoconvex complex Finsler metric on $E^*$, one has $\omega_{FS}(G_0)>0$ (means positive along fibers). Now we consider the following flow over the projective bundle $P(E^*)$: 
\begin{align}\label{flow0}
\begin{cases}
&\frac{\p}{\p t}\log G=\Delta_{\Omega}\log G,\\
& \omega_{FS} (G)>0,\\
&G(0)=G_0.	
\end{cases}
\end{align}
Here $\Delta_{\Omega}:=\sqrt{-1}\Lambda \p\b{\p}$, $\Lambda$ is the adjoint operator of $\Omega\wedge \bullet$. 

One can also define a horizontal and real $(1,1)$-form $T$ on $P(E^*)$ as follow,
\begin{align}\label{defnT}
(-\sqrt{-1})T(u,\o{u}):=\langle \sqrt{-1}R^g(u,\o{u}), -\Psi\rangle_{\Omega}-\left|i_u\p^V\Psi\right|^2_{\Omega}	
\end{align}
for any horizontal vector $u=u^\alpha\frac{\delta}{\delta z^{\alpha}}$, where $\langle \sqrt{-1}R^g(u,\o{u}), -\Psi\rangle_{\Omega}:=(-\Psi)_{\alpha\b{\delta}}g^{\alpha\b{\beta}}g^{\gamma\b{\delta}}R^g_{\gamma\b{\beta}\sigma\b{\tau}}u^\sigma\b{u}^\tau$ and $\left|i_u\p^V\Psi\right|^2_{\Omega}:=(\log G)^{a\b{b}}\p_a\Psi_{\alpha\b{\beta}}\o{\p_b\Psi_{\gamma\b{\tau}}}u^\alpha\o{u}^{\gamma}g^{\tau\b{\beta}}$, $R^g$ denotes the Chern curvature of $\omega$. 
Now we assume that $T$ satisfies the {\it the null eigenvector assumption} (see Theorem \ref{thm1}), namely $(-\sqrt{-1})T(U,\o{U})\geq 0$ whenever $\p\b{\p}\log G(U,\o{U})\geq 0$ and $i_U (\p\b{\p}\log G)=0$ for a $(1,0)$-type vector $U$ of $TP(E^*)$.
By using the maximum principle for real $(1,1)$-forms, we obtain
\begin{thm}[=Theorem \ref{thm2}]\label{thm0}
Let $\pi: (E^*, G_0)\to M$ be a  holomorphic Finsler vector bundle over the compact K\"ahler manifold $M$ with $\sqrt{-1}\p\b{\p}\log G_0\geq 0$.
If the horizontal $(1,1)$-form $T$ satisfies the null eigenvector assumption, then 
$
\sqrt{-1}\p\b{\p}\log G(t)\geq 0
$ along the flow (\ref{flow0}) for all $t\geq 0$ such that the solution exists.
\end{thm}
 
 In this paper, we will give two applications of Theorem \ref{thm0}.
 
For the first application,  we consider the case of curve, i.e. $\dim M=1$. 
In this case, any Hermitian metric  on $M$ is K\"ahler automatically, and one can prove that the $(1,1)$-form $T$ satisfies 
 the null eigenvector assumption.
By Theorem \ref{thm0}, we obtain
\begin{prop}[=Proposition \ref{prop3}]\label{prop0}
If $M$ is a curve, then the semipositivity of the curvature of $\mc{O}_{P(E^*)}(1)$ is preserved along the flow (\ref{flow0}).	
\end{prop}
In particular, if $G_0=h_0^{i\b{j}}v_i\b{v}_j$ comes from a Hermitian metric $(h_0^{i\b{j}})$ of $E^*$ and 
\begin{align}
\Omega=p^*\omega+\omega_{FS}(G)	
\end{align}
for a fixed Hermitian metric $
\omega$, by Remark \ref{rem2} (1), (\ref{flow0}) is equivalent to the following Hermitian-Yang-Mills flow:
\begin{align}\label{HYM1}
\begin{cases}
	&h^{-1}\cdot \frac{\p h}{\p t}+\Lambda R^h+(r-1)I=0,\\
	&(h_{i\b{j}}(t))>0,\\
	&h_{i\b{j}}(0)=(h_0)_{i\b{j}}.
	\end{cases}
\end{align}
Recall that
\begin{defn}[\cite{Gri}]\label{defn2}
The Chern curvature of the metric $(h_{i\b{j}})$ is called {\it Griffiths positive} (resp. {\it Griffiths semipositive}) if 	
\begin{align*}
R_{i\b{j}\alpha\b{\beta}}X^i\o{X^j}Y^{\alpha}\o{Y^{\beta}}>0 \quad (\text{resp.}\geq 0)	
\end{align*}
for any two nonzero vectors $X=X^ie_i\in E$ and $Y=Y^{\alpha}\frac{\p}{\p z^{\alpha}}\in TM$. Here $R_{i\b{j}\alpha\b{\beta}}=-\frac{\p^2 h_{i\b{j}}}{\p z^\alpha\p\b{z}^\beta}+h^{\b{l}k}\frac{\p h_{i\b{l}}}{\p z^\alpha}\frac{\p h_{k\b{j}}}{\p\b{z}^\beta}$ denotes the Chern curvature of $(h_{i\b{j}})$.
For the case of $E=TM$, the Hermitian metric $(h_{i\b{j}})$ is called { \it has positive (nonnegative) holomorphic bisectional curvature} if its Chern curvature is Griffiths positive (semipositive).
\end{defn}
 By Proposition \ref{prop0}, we have
\begin{cor}[=Corollary \ref{cor3.2}]\label{cor0}
	If $M$ is a curve, then the Griffiths semipositivity  is preserved along the Hermitian-Yang-Mills flow (\ref{HYM1}).
\end{cor}
For the second application, we assume that $E=TM$ and take
\begin{align*}
\omega(G)=\sqrt{-1}g_{\alpha\b{\beta}}dz^{\alpha}\wedge d\b{z}^{\beta},
\end{align*}
where $(g_{\alpha\b{\beta}})$ denotes the inverse of the matrix  $\left(\frac{\p^2 G}{\p v_{\alpha}\p\b{v}_{\beta}}\right)$ (see Section \ref{sec1} for the definition of $\frac{\p^2 G}{\p v_{\alpha}\p\b{v}_{\beta}}$). Let $G_0=g_0^{\alpha\b{\beta}}v_\alpha\b{v}_\beta$ be the strongly pseudoconvex complex Finsler metric on $T^*M$ induced by the following K\"ahler metric 
\begin{align*}
\omega_0=\sqrt{-1}(g_0)_{\alpha\b{\beta}}dz^{\alpha}\wedge d\b{z}^{\beta}.
\end{align*}
Then the flow (\ref{flow0}) is equivalent to the following K\"ahler-Ricci flow
\begin{align}\label{KR0}
\begin{cases}
	& \frac{\p\omega}{\p t}+\text{Ric}(\omega)+(n-1)\omega=0,\\
	&\omega>0,\\
	&\omega(0)=\omega_0.
\end{cases}	
\end{align}
The solution of (\ref{flow0}) is induced from the K\"ahler metric $\omega=\sqrt{-1}g_{\alpha\b{\beta}}dz^{\alpha}\wedge d\b{z}^{\beta}$. In this case, the $(1,1)$-form $T$ can also be proved satisfying the null eigenvector assumption (see \cite[Page 254, Claim 2.2]{Bou}).
From Theorem \ref{thm0} and Definition \ref{defn2}, we can reprove the following Mok's proposition, which is contained in \cite[Proposition 1.1]{Mok} (see also \cite[Theorem 5.2.10]{Bou}).
\begin{prop}[{\cite[Proposition 1.1]{Mok}}]\label{prop0.1}
	If $(M, \omega_0)$ is a compact K\"ahler manifold with nonnegative holomorphic bisectional curvature, then the nonnegativity is preserved along the K\"ahler-Ricci flow (\ref{KR0}).
\end{prop}
\begin{rem}
	By (\ref{HYM1}) and (\ref{KR0}), the flow (\ref{flow0}) is a natural generalization of both Hermitian-Yang-Mills flow and K\"ahler-Ricci flow. And there are some other flows, which are also the generalizations of the K\"ahler-Ricci flow. For example, Gill \cite{Gill} introduced the {\it Chern-Ricci flow} on Hermitian manifolds, and many  properties of the flow were established in \cite{To1, To2}. Especially,  Yang \cite{Yang}  proved the nonnegativity of the holomorphic bisectional curvature is not necessarily preserved under the Chern-Ricci flow. In \cite{Streets}, Streets and Tian introduced the {\it Hermitian curvature flow}, proved short time existence for this flow, and derive basic long time blowup and regularity results, and in \cite{Streets1, Streets2} they introduced a parabolic flow of pluriclosed metrics and obtained some regularity results for solutions to this equation. For a particular version of the Hermitian curvature flow, Ustinovskiy \cite{Yury} proved the the property of Griffiths positive (nonnegative) Chern curvature is preserved along this flow. 
\end{rem}
This article is organized as follows. In Section \ref{sec1},
 we shall fix the notation and recall some basic definitions and facts on complex Finsler vector bundles, Griffiths positivity (semipositivity), and maximal principle for real $(1,1)$-forms. In Section \ref{sec2}, we will define a flow over projective bundle $P(E^*)$ and study the positivity along this flow, Theorem \ref{thm0} would be proved in this section. In Section \ref{sec3}, we will give two applications of Theorem \ref{thm0}, and we will prove Proposition \ref{prop0}, Corollary \ref{cor0} and Proposition \ref{prop0.1}.

\vspace{5mm}
{\bf Acknowledgements.} The author would like to thank Professor Kefeng Liu and Professor Huitao Feng for their guidance over the years, 
and thank Professor Xiaokui Yang for many valuable discussions. The author would like to thank the anonymous referees for valuable
comments which helped to improve the paper.


\section{Preliminaries}\label{sec1}

In this section, we shall fix the notation and recall some basic definitions and facts on complex Finsler vector bundles, Griffiths positivity (semipositivity), and maximal principle for real $(1,1)$-forms.  For more details we refer to \cite{AP, Aikou, Chow, Cao-Wong, FLW, FLW2, FLW1, Gri, Ko1, Liu1, Munteanu, Wan}.
\subsection{Complex Finsler vector bundle} 
Let $M$ be a compact complex manifold of dimension $n$, and let $\pi:E\to M$ be a holomorphic vector bundle of rank $r$ over $M$.  Let $z=(z^1,\cdots, z^n)$ be a local coordinate system  in $M$, and $\{e_i\}_{1\leq i\leq r}$ be a local holomorphic frame of $E$. With respect to the local frame of $E$, an element of $E$ can be written as
$$v=v^i e_i\in E,$$
where we adopt the summation convention of Einstein. In this way, one gets a local coordinate system of the complex manifold $E$: 
\begin{align}\label{coor}
(z;v)=(z^1,\cdots,z^n; v^1,\cdots, v^r).
\end{align}

\begin{defn}[\cite{Ko1}]\label{defn1}
A Finsler metric  $G$ on the holomorphic vector bundle $E$ is a continuous function $G:E\to\mathbb{R}$ satisfying the following conditions:
\begin{description}
  \item [F1)] $G$ is smooth on $E^o=E\setminus O$, where $O$ denotes the zero section of $E$;
  \item[F2)] $G(z,v)\geq 0$ for all $(z,v)\in E$ with $z\in M$ and $v\in\pi^{-1}(z)$, and $G(z,v)=0$ if and only if $v=0$;
  \item[F3)] $G(z,\lambda v)=|\lambda|^2G(z,v)$ for all $\lambda\in\mathbb{C}$.
\end{description}
Moreover, $G$ is called strongly pseudoconvex if
\begin{description}
  \item[F4)] the Levi form ${\sqrt{-1}}\partial\bar\partial G$ on $E^o$ is positive-definite along each fiber $E_z=\pi^{-1}(z)$ for $z\in M$.
\end{description}
\end{defn}
Clearly, for any Hermitian metric on $E$, one can associate it with a strongly pseudoconvex complex Finsler metric on $E$.

We write
\begin{align*}
\begin{split}
& G_{i}=\partial G/\partial v^{i},\quad G_{\bar j}=\partial G/\partial\bar{v}^{j},\quad G_{i\bar{j}}=\partial^{2}G/\partial v^{i}\partial\bar{v}^{j},\\
& G_{i\alpha}=\partial^{2}G/\partial v^{i}\partial z^{\alpha}, \quad G_{i\bar j\bar\beta}=\partial^3G/\partial v^{i}\partial\bar v^j\partial\bar z^{\beta},\quad etc.,
\end{split}
\end{align*}
to denote the differentiation with respect to $v^i,\bar v^j$ ($1\leq i,j\leq r$), $z^\alpha,\bar z^\beta$ ($1\leq\alpha,\beta\leq n$). In the following lemma we collect some useful identities related to a Finsler metric $G$. 
\begin{lemma}[\cite{Cao-Wong, Ko1}]\label{1.1} The following identities hold for any $(z,v)\in E^o$, $\lambda\in \mathbb{C}$:
\begin{align*}
G_i(z,\lambda v)=\bar\lambda G_i(z,v),\quad G_{i\bar j}(z,\lambda v)=G_{i\bar j}(z,v)=\o{G_{j\bar i}(z,v)};
\end{align*}
\begin{align*}
G_i(z,v)v^i=G_{\bar j}(z,v)\bar v^j=G_{i\bar j}(z,v)v^i\bar v^j=G(z,v);
\end{align*}
\begin{align*}
G_{ij}(z,v)v^i=G_{i\bar j k}(z,v)v^i=G_{i\bar j\bar k}(z,v)\bar v^j=0.
\end{align*}
\end{lemma}

If $G$ is a strongly pseudoconvex complex Finsler metric on $M$, then there is a canonical h-v decomposition of the holomorphic tangent bundle $TE^o$ of $E^o$ (see \cite[\S 5]{Cao-Wong} or \cite[\S 1]{FLW}).
\begin{align*}
TE^o=\mc{H}\oplus \mc{V}.
\end{align*}
In terms of local coordinates,
\begin{align*}
\mc{H}=\text{span}_{\mb{C}}\left\{\frac{\delta}{\delta z^\alpha}=\frac{\p}{\p z^\alpha}-G_{\alpha\b{j}}G^{\b{j}k}\frac{\p}{\p v^k}, 1\leq \alpha\leq n\right\},\quad \mc{V}=\text{span}_{\mb{C}}\left\{\frac{\p}{\p v^i}, 1\leq i\leq r\right\}.
\end{align*}
The dual bundle $T^*E^o$ also has a smooth h-v decomposition $T^*E^o=\mc{H}^*\oplus\mc{V}^*$:
\begin{align}\label{H}
\mc{H}^*=\text{span}_{\mb{C}}\{dz^{\alpha}, 1\leq \alpha\leq n\},\quad \mc{V}^*=\text{span}_{\mb{C}}\{\delta v^i=dv^i+G^{\b{j}i}G_{\alpha\b{j}}dz^\alpha,\quad 1\leq i\leq r\}.
\end{align}
Moreover, the differential operators
\begin{align}\label{HV}
\p^V=\frac{\p}{\p v^i}\otimes \delta v^i,\quad \p^H=\frac{\delta}{\delta z^{\alpha}}\otimes dz^{\alpha}.
\end{align}
are well-defined.

With respect to the h-v decomposition (\ref{H}), the $(1,1)$-form $\sqrt{-1}\p\b{\p}\log G$ has the following decomposition. For readers' convenience, we give a proof of the following lemma due to Kobayashi and Aikou (cf. \cite{Ko1, Aikou}).
\begin{lemma}[\cite{Ko1, Aikou}]\label{lemma1}
Let $G$ be a strongly pseudoconvex complex Finsler metric on $E$. One has
\begin{align*}
\sqrt{-1}\p\b{\p}\log G=-\Psi+\omega_{V},
\end{align*}
where $\Psi$ is called the Kobayashi curvature (see \cite[Definition 1.2]{FLW}),  
\begin{align}\label{Psi}
\Psi=\sqrt{-1}R_{i\b{j}\alpha\b{\beta}}\frac{v^i\b{v}^j}{G}dz^\alpha\wedge d\b{z}^\beta,\quad \omega_{V}=\sqrt{-1}\frac{\p^2 \log G}{\p v^i\p\b{v}^j}\delta v^i\wedge \delta\b{v}^j,
\end{align}
with
\begin{align*}
	R_{i\b{j}\alpha\b{\beta}}=-\frac{\p^2 G_{i\b{j}}}{\p z^\alpha\p\b{z}^\beta}+G^{\b{l}k}\frac{\p G_{i\b{l}}}{\p z^\alpha}\frac{\p G_{k\b{j}}}{\p\b{z}^\beta}.
\end{align*}
\end{lemma}
\begin{proof}
By (\ref{H}), we have
\begin{align}\label{1.5}
\begin{split}
&\quad \frac{\partial^{2}\log G}{\partial v^{i}\partial\bar{v}^{j}}\delta v^{i}\wedge\delta\bar{v}^{j}=\frac{\partial^{2}\log G}{\partial v^{i}\partial\bar{v}^{j}}(dv^i+G_{\alpha\b{l}}G^{\b{l}i}dz^{\alpha})\wedge(d\b{v}^j+G_{\b{\beta}k}G^{\b{j}k}d\b{z}^{\beta})\\
&=\frac{\partial^{2}\log G}{\partial v^{i}\partial\bar{v}^{j}}d v^{i}\wedge d\bar{v}^{j}+\frac{\partial^{2}\log G}{\partial v^{i}\partial\bar{v}^{j}}G_{\b{\beta}k}G^{\b{j}k}dv^i\wedge d\bar{z}^{\beta}\\
&\quad +\frac{\partial^{2}\log G}{\partial v^{i}\partial\bar{v}^{j}}G_{\alpha\b{l}}G^{\b{l}i}dz^{\alpha}\wedge d\b{v}^j
+\frac{\p^2\log G}{\p v^i\p\b{v}^j}G_{\alpha\b{l}}G^{\b{l}i}G_{k\b{\beta}}G^{\b{j}k}dz^{\alpha}\wedge d\b{z}^{\beta}.
\end{split}
\end{align}
For the last three terms in the RHS of (\ref{1.5}), one has from Lemma \ref{1.1}
\begin{align}\label{1.8}
	\frac{\partial^{2}\log G}{\partial v^{i}\partial\bar{v}^{j}}G_{\b{\beta}k}G^{\b{j}k}dv^i\wedge d\bar{z}^{\beta}=\frac{GG_{i\b{j}}-G_iG_{\b{j}}}{G^2}G_{\b{\beta}k}G^{\b{j}k}dv^i\wedge d\bar{z}^{\beta}=\frac{\p^2\log G}{\p v^i\p\b{z}^{\beta}}dv^i\wedge d\b{z}^{\beta},
\end{align}
\begin{align}\label{1.6}
\frac{\partial^{2}\log G}{\partial v^{i}\partial\bar{v}^{j}}G_{\alpha\b{l}}G^{\b{l}i}dz^{\alpha}\wedge d\b{v}^j=\frac{GG_{i\b{j}}-G_iG_{\b{j}}}{G^2}G_{\alpha\b{l}}G^{\b{l}i}dz^{\alpha}\wedge d\b{v}^j=\frac{\p^2\log G}{\p z^{\alpha}\p\b{v}^j}dz^{\alpha}\wedge d\b{v}^j
\end{align}
and 
\begin{align}\label{1.7}
\begin{split}
	\frac{\p^2\log G}{\p v^i\p\b{v}^j}G_{\alpha\b{l}}G^{\b{l}i}G_{k\b{\beta}}G^{\b{j}k}dz^{\alpha}\wedge d\b{z}^{\beta}&=\frac{GG_{i\b{j}}-G_iG_{\b{j}}}{G^2}G_{\alpha\b{l}}G^{\b{l}i}G_{k\b{\beta}}G^{\b{j}k}dz^{\alpha}\wedge d\b{z}^{\beta}\\
	&=\frac{1}{G^2}(G G^{\b{l}k}G_{\alpha\b{l}}G_{k\b{\beta}}-G_{\alpha}G_{\b{\beta}})dz^{\alpha}\wedge d\b{z}^{\beta}.
	\end{split}
\end{align}

Submitting (\ref{1.8}), (\ref{1.6}) and (\ref{1.7}) into (\ref{1.5}), we obtain
\begin{align*}
	\frac{\partial^{2}\log G}{\partial v^{i}\partial\bar{v}^{j}}\delta v^{i}\wedge\delta\bar{v}^{j}&=\frac{\partial^{2}\log G}{\partial v^{i}\partial\bar{v}^{j}}d v^{i}\wedge d\bar{v}^{j}+\frac{\p^2\log G}{\p v^i\p\b{z}^{\beta}}dv^i\wedge d\b{z}^{\beta}\\
	&\quad+\frac{\p^2\log G}{\p z^{\alpha}\p\b{v}^j}dz^{\alpha}\wedge d\b{v}^j+\frac{1}{G^2}(G G^{\b{l}k}G_{\alpha\b{l}}G_{k\b{\beta}}-G_{\alpha}G_{\b{\beta}})dz^{\alpha}\wedge d\b{z}^{\beta}\\
	&=\p\b{\p}\log G+\frac{1}{G}(G^{\b{l}k}G_{\alpha\b{l}}G_{k\b{\beta}}-G_{\alpha\b{\beta}})dz^{\alpha}\wedge d\b{z}^{\beta}\\
	&=\p\b{\p}\log G+(G^{\b{l}k}G_{\alpha i\b{l}}G_{k\b{j}\b{\beta}}-G_{i\b{j}\alpha\b{\beta}})\frac{v^i\b{v}^j}{G}dz^{\alpha}\wedge d\b{z}^{\beta}\\
	&=\p\b{\p}\log G-\sqrt{-1}\Psi,
\end{align*}
which completes the proof.
	
\end{proof}

 Let $q$ denote the natural projection 
\begin{align*}
q: E^o\to P(E):=E^o/\mb{C}^*\quad (z; v)\mapsto (z;[v]):=(z^1,\cdots, z^n; [v^1,\cdots, v^r]),	
\end{align*}
which gives a local coordinate system of $P(E)$ by 
\begin{align}\label{2.4}
(z;w):=(z^1,\cdots, z^n; w^1,\cdots, w^{r-1})=\left(z^1,\cdots,z^n; \frac{v^1}{v^k},\cdots,\frac{v^{k-1}}{v^k},\frac{v^{k+1}}{v^{k}},\cdots, \frac{v^r}{v^k}\right)
\end{align}
on $U_k:=\{(z,[v])\in P(E), v^k\neq 0\}$.

Denote by $((\log G)^{\b{b}a})_{1\leq a,b\leq r-1}$  the inverse of the matrix  $\left((\log G)_{a\b{b}}:=\frac{\p^2\log G}{\p w^a\b{w}^b}\right)_{1\leq a,b\leq r-1}$ and set
\begin{align*}
\delta w^a=d w^a+(\log G)_{\alpha\b{b}}(\log G)^{\b{b}a}dz^{\alpha}.	
\end{align*}

One can define a vertical $(1,1)$-form on $P(E)$ by 
\begin{align}\label{vertical form}
\omega_{FS}(G):=\sqrt{-1}\frac{\p^2\log G}{\p w^a\p\b{w}^b}\delta w^a\wedge \delta\b{w}^b,	
\end{align}
which is well-defined (see e.g. \cite[Section 1]{FLW1}). Moreover, 
\begin{lemma}\label{lemma2}
By the pullback $q^*: A^{1,1}(P(E))\to A^{1,1}(E^o)$, one has
\begin{align}\label{1.9}
q^*\omega_{FS}(G)=\omega_{V}.	
\end{align}
\end{lemma}
\begin{proof}
 We only need to prove (\ref{1.9}) at one point $(z,[v])\in U_k$. Without loss of generality, we assume that $k=r$. For any point $(z,[v])\in U_r$, one has from (\ref{2.4})
 \begin{align}\label{2.5}
 	q_*\left(\frac{\p}{\p v^r}\right)=-\sum_{a=1}^{r-1}\frac{v^{a}}{(v^r)^2}\frac{\p}{\p w^{a}},\quad  q_*\left(\frac{\partial}{\partial v^{b}}\right)=\frac{1}{v^{r}}\frac{\partial}{\partial w^{b}},\quad 1\leq b\leq r-1.
 \end{align}

By (\ref{2.5}), one has
\begin{align}\label{1.144}
\begin{split}
\frac{\p^2\log G}{\p w^a\p\b{w}^b}&=(\p\b{\p}\log G)(\frac{\p}{\p w^a},\frac{\p}{\p \b{w}^b})\\
&=(\p\b{\p}\log G)(q_*(v^r\frac{\p}{\p v^a}),q_*(\b{v}^r\frac{\p}{\p \b{v}^b}))\\
&=q^*(\p\b{\p}\log G)(v^r\frac{\p}{\p v^a},\b{v}^r\frac{\p}{\p \b{v}^b})\\
&=|v^r|^2\frac{\p^2\log G}{\p v^a\p\b{v}^b}.
\end{split}
\end{align}
Similarly, 
\begin{align}\label{1.11}
\frac{\p^2\log G}{\p z^{\alpha}\p\b{w}^b}=\b{v}^r\frac{\p^2\log G}{\p z^{\alpha}\p\b{v}^b},\quad 	\frac{\p^2\log G}{\p w^a\p\b{z}^{\beta}}=v^r\frac{\p^2\log G}{\p v^a\p\b{z}^{\beta}}
\end{align}
and
\begin{align}\label{1.12}
\frac{\p^2\log G}{\p v^a\p\b{v}^r}=-\frac{1}{|v^r|^2}\frac{\b{v}^b}{\b{v}^r}\frac{\p^2\log G}{\p w^a\p\b{w}^b},\quad \frac{\p^2\log G}{\p v^r\p\b{v}^b}=-\frac{1}{|v^r|^2}\frac{v^a}{v^r}\frac{\p^2\log G}{\p w^a\p\b{w}^b},\quad \frac{\p^2\log G}{\p v^r\p\b{v}^r}=\frac{v^a\b{v}^b}{|v^r|^4}\frac{\p^2\log G}{\p w^a\p\b{w}^b}.	
\end{align}

By a direct checking, one has
  \begin{align}\label{1.10}
  (\log G)^{\bar{b}a}=\frac{G}{|v^{r}|^{2}}\left(-\frac{v^{a}}{v^{r}}G^{\bar{b}r}+G^{\bar{b}a}
  +\frac{\bar{v}^{b}v^{a}}{|v^{r}|^{2}}G^{\bar{r}r}
  -\frac{\bar{v}^{b}}{\bar{v}^{r}}G^{\bar{r}a}\right).\end{align}
By (\ref{1.11}) and (\ref{1.10}), we have
\begin{align*}
(\log G)_{a\b{\beta}}(\log G)^{\b{b}a}=\left(\frac{1}{\b{v}^r}G^{\b{b}i}-\frac{\b{v}^b}{(\b{v}^r)^2}G^{\b{r}i}\right)(G_{i\b{\beta}}-\frac{1}{G}G_{\b{\beta}}G_i)=\frac{1}{\b{v}^r}G^{\b{b}i}G_{i\b{\beta}}-\frac{\b{v}^b}{(\b{v}^r)^2}G^{\b{r}i}G_{i\b{\beta}}.	
\end{align*}
So
\begin{align}\label{1.13}
\begin{split}
	q^*(\delta\b{w}^b)&=q^*(d\b{w}^b+(\log G)_{a\b{\beta}}(\log G)^{\b{b}a}d\b{z}^{\beta})\\
	&=\frac{1}{\b{v}^r}(d\b{v}^b+G^{\b{b}i}G_{i\b{\beta}}d\b{z}^{\beta})-\frac{\b{v}^b}{(\b{v}^r)^2}(d\b{v}^r+G^{\b{r}i}G_{i\b{\beta}}d\b{z}^{\beta})\\
	&=\frac{1}{\b{v}^r}\delta \b{v}^b-\frac{\b{v}^b}{(\b{v}^r)^2}\delta\b{v}^r.
\end{split}	
\end{align}
From (\ref{1.144}), (\ref{1.12}) and (\ref{1.13}), we obtain
\begin{align*}
	\begin{split}
		q^*\omega_{FS}(G)&=q^*\left(\sqrt{-1}\frac{\p^2\log G}{\p w^a\p\b{w}^b}\delta w^a\wedge \delta\b{w}^b\right)\\
		&=\sqrt{-1}|v^r|^2\frac{\p^2\log G}{\p v^a\p\b{v}^b}\left(\frac{1}{v^r}\delta v^a-\frac{v^a}{(v^r)^2}\delta v^r\right)\left(\frac{1}{\b{v}^r}\delta \b{v}^b-\frac{\b{v}^b}{(\b{v}^r)^2}\delta\b{v}^r\right)\\
		&=\sqrt{-1}\frac{\p^2\log G}{\p v^a\p\b{v}^b}\left(\delta v^a\wedge \delta \b{v}^b-\frac{\b{v}^b}{\b{v}^r}\delta v^a\wedge \delta\b{v}^r-\frac{v^a}{v^r}\delta v^r\wedge \delta\b{v}^b+\frac{v^a\b{v}^b}{|\b{v}^r|^2}\delta v^r\wedge \delta\b{v}^r\right)\\
		&=\sqrt{-1}\sum_{i,j=1}^r\frac{\p^2\log G}{\p v^i\p\b{v}^j}\delta v^i\wedge \delta\b{v}^j=\omega_V.
	\end{split}
\end{align*}
\end{proof}
\begin{rem}
\begin{itemize}\label{rem1}
\item[(1)] By (\ref{2.5}) and (\ref{1.10}), we get
\begin{align*}
\begin{split}
q_*(\frac{\delta}{\delta z^{\alpha}})&=q_*(\frac{\p}{\p z^{\alpha}}-G_{\alpha\b{l}}G^{\b{l}k}\frac{\p}{\p v^k})\\
&=\frac{\p}{\p z^{\alpha}}-(\frac{1}{v^r}G_{\alpha\b{l}}G^{\b{l}a}-\frac{v^a}{(v^r)^2}G_{\alpha\b{l}}G^{\b{l}r})\frac{\p}{\p w^a}\\
&=\frac{\p}{\p z^\alpha}-(\log G)_{\alpha\b{b}}(\log G)^{\b{b}a}\frac{\p}{\p w^a}.	
\end{split}
\end{align*}
For convenience, we will identify $\frac{\delta}{\delta z^{\alpha}}$ with $q_*(\frac{\delta}{\delta z^{\alpha}})$, and denote $N_{\alpha}^a:=(\log G)_{\alpha\b{b}}(\log G)^{\b{b}a}$, so 
\begin{align}\label{horizontal frame}
\frac{\delta}{\delta z^{\alpha}}=\frac{\p}{\p z^{\alpha}}-N^a_{\alpha}\frac{\p}{\p w^a}.	
\end{align}

\item[(2)] For any smooth function $f\in C^{\infty}(P(E))$, the vertical Laplacian is defined by 
\begin{align*}
\Delta^V f:=(\log G)^{\b{b}a}\frac{\p^2}{\p w^a\p\b{w}^b}f.	
\end{align*}
By identifying $f$ with $q^*f$, one has from (\ref{2.5}) and (\ref{1.10})
\begin{align*}
\Delta^V f=(\log G)^{\b{b}a}\frac{\p^2}{\p w^a\p\b{w}^b}f=GG^{i\b{j}}\frac{\p^2}{\p v^i\p\b{v}^j}f.	
\end{align*}
\item[(3)] For the Hermitian metric 
$$\Omega=\omega(G)+\omega_{FS}(G)$$
on $P(E)$,
where $\omega(G)=\sqrt{-1}g_{\alpha\b{\beta}}dz^\alpha\wedge d\b{z}^\beta$, then one can define the {\it horizontal laplacian} by 
\begin{align}
\Delta^H f:=g^{\alpha\b{\beta}}(\p\b{\p}f)(\frac{\delta}{\delta z^\alpha},\frac{\delta}{\delta\b{z}^\beta})	
\end{align}
for any smooth function $f\in C^{\infty}(P(E))$.
\item[(4)] Note that $\frac{\sqrt{-1}}{2\pi}\p\b{\p}\log G$ is a $(1,1)$-form on $P(E)$, which represents the first Chern class $c_1(\mc{O}_{P(E)}(1))$. And $\Psi$ is also a $(1,1)$-form on $P(E)$, combining Lemma \ref{lemma1} with Lemma \ref{lemma2}, one has when restricting to $P(E)$
\begin{align}\label{1.14}
	\sqrt{-1}\p\b{\p}\log G=-\Psi+\omega_{FS}(G).
\end{align}
\end{itemize}
\end{rem}

\begin{prop}\label{prop1}
A Finsler metric $G$ is strongly pseudoconvex if and only if $\omega_{FS}(G)$ is positive along each fiber of $p: P(E)\to M$. 
\end{prop}
\begin{proof}
	By Definition \ref{defn1}, $G$ is strongly pseudoconvex if $(G_{i\b{j}})$ is a positive definite matrix, which gives a inner product $\langle\cdot,\cdot\rangle$ on the vertical subbundle $\mc{V}$. 
	
	Denote $T=v^i\frac{\p}{\p v^i}$. If $G$ is strongly pseudoconvex, for any $X=X^i\frac{\p}{\p v^i}$, then 
	\begin{align*}
	\begin{split}
	(-\sqrt{-1})\omega_{V}(X,\b{X}) &=\frac{1}{G^2}(GG_{i\b{j}}-G_iG_{\b{j}})X^i\b{X}^j\\
	&=\frac{1}{G^2}(\|X\|^2\|T\|^2-|\langle X,T\rangle|^2)\geq 0,	
	\end{split}
	\end{align*}
the equality holds if and only if $X=\lambda T$ for some constant $\lambda\in\mb{C}$. So $\omega_V$ has $r-1$ positive eigenvalues and one zero eigenvalue. Since $\omega_V(T,\o{T})=0$ and $q_*(T)=0$, by Lemma \ref{lemma2}, $\omega_{FS}(G)$ is positive along each fiber of $p: P(E)\to M$.   

Conversely, if $\omega_{FS}(G)$ is positive along each fiber, then $\omega_V=q^*\omega_{FS}(G)$ has $r-1$ positive eigenvalues and one zero eigenvalue, and $\omega_V(T,\o{T})=0$. So
\begin{align*}
G_{i\b{j}}X^i\b{X}^j=\frac{1}{G}|G_i X^i|^2+G(-\sqrt{-1})\omega_V(X,\o{X})\geq 0.	
\end{align*}
 Moreover, $G_{i\b{j}}X^i\b{X}^j=0$ if and only if $X=\lambda T$ and $G_iX^i=0$, if and only if $\lambda=0$. So $(G_{i\b{j}})$ is a positive definite matrix. 
\end{proof}

Let $(h_{i\b{j}})$ be a Hermitian metric on $E$ with respect to a local holomorphic frame $\{e_i\}_{1\leq i\leq r}$. 
The Hermitian metric $(h_{i\b{j}})$ induces a strongly pseudoconvex complex Finsler metric on $E^*$ by
\begin{align*}
G:=h^{i\b{j}}v_i\b{v}_j.	
\end{align*}
By Remark \ref{rem1}, we have the following decomposition 
\begin{align*}
\sqrt{-1}\p\b{\p}\log G=-\Psi+\omega_{FS}(G),	
\end{align*}
where 
\begin{align*}
-\Psi=-R^{i\b{j}}_{~~\alpha\b{\beta}}\frac{v_i\b{v}_j}{G}\sqrt{-1}dz^{\alpha}\wedge d\b{z}^{\beta}=R_{k\b{l}\alpha\b{\beta}}\frac{G^{\b{l}i}v_iG^{\b{j}k}\b{v}_j}{G}\sqrt{-1}dz^{\alpha}\wedge d\b{z}^{\beta}.	
\end{align*}

From Proposition \ref{prop1}, $\omega_{FS}(G)$ is positive along fibers.  So
\begin{prop}\label{prop2}
	The Chern curvature of $(h_{i\b{j}})$ is Griffiths positive (resp. semipositive)if and only if $\sqrt{-1}\p\b{\p}\log G$ is a positive (resp. semipositive) $(1,1)$-form on $P(E^*)$.
\end{prop}
 
\subsection{Maximum principle for real $(1,1)$-forms} 

In this subsection, we recall the maximum principle for real $(1,1)$-forms. For more details, one can refer to \cite{Bando, Chow, Mok, Niu}. The following version maximum principle is same as the tensor maximum
    principle \cite[Theorem 4.6]{Chow}, so we omit its proof.  
    \begin{thm}[{\cite[Theorem 4.6]{Chow}}]\label{thm1}
Let $\omega(t)=\sqrt{-1}g_{\alpha\b{\beta}}(t)dz^{\alpha}\wedge d\b{z}^{\beta}$	be a smooth $1$-parameter family of Hermitian metrics on a compact complex manifold $M$. Let $\eta(t)=\sqrt{-1}\eta_{\alpha\b{\beta}}(t)dz^{\alpha}\wedge d\b{z}^{\beta}$ be a real $(1,1)$-form satisfying 
\begin{align*}
\frac{\p}{\p t}\eta\geq \Delta_{\omega(t)}\eta+\sigma,
\end{align*}
where $\Delta_{\omega(t)}:=g^{\alpha\b{\beta}}\n_{\frac{\p}{\p z^{\alpha}}}\n_{\frac{\p}{\p\b{z}^{\beta}}}$, $\n$ denotes the Chern connection of $\omega(t)$, $\sigma(\omega,t)$ is a real $(1,1)$-form which is locally  Lipschitz in all its arguments and satisfies the {\bf null eigenvector assumption} that 
$$(-\sqrt{-1})\sigma(V,\o{V})(z,t)=(\sigma_{\alpha\b{\beta}}V^\alpha\o{V^\beta})(z,t)\geq 0$$
whenever $V(z,t)=V^{\alpha}\frac{\p}{\p z^{\alpha}}$ is a null eigenvector of $\eta(t)$, that is whenever
$$(\eta_{\alpha\b{\beta}}V^{\alpha})(z,t)=0.$$
If $\eta(0)\geq 0$, then $\eta(t)\geq 0$ for all $t\geq 0$ such that the solution exists.
\end{thm}

\section{A flow over projective bundle}\label{sec2}

\subsection{Definition of the flow}
Let $M$ be a compact complex manifold of dimension $n$. Let $\pi:E\to M$ be a holomorphic vector bundle of rank $r$ over $M$. Let $E^*$ denote the dual bundle of $E$, $\{e^i\}_{i=1}^r$ be a local holomorphic frame of $E^*$. Then 
\begin{align*}
(z, v)=(z^1,\cdots, z^n; v_1,\cdots, v_r)	
\end{align*}
gives a local coordinate system of the complex manifold $E^*$, which represents the point $v^ie_i\in E^*$. For any strongly pseudoconvex complex Finsler metric $G$ on $E^*$, by Remark \ref{rem1}, we have the following decomposition
\begin{align*}
\sqrt{-1}\p\b{\p}\log G=-\Psi+\omega_{FS}(G).	
\end{align*}
By Proposition \ref{prop1}, $\omega_{FS}(G)$ is positive along each fiber of $p: P(E^*)\to M$.

Let 
\begin{align*}
\omega(G)=\sqrt{-1}g_{\alpha\b{\beta}}(G)dz^{\alpha}\wedge d\b{z}^{\beta} 	
\end{align*}
be a horizontal $(1,1)$-form on $P(E^*)$ depending smoothly on the Finsler metric $G$, which is positive on horizontal directions, namely $(g_{\alpha\b{\beta}}(G))$ is a positive definite matrix. 

Then one can define a Hermitian metric on $P(E^*)$ by
\begin{align*}
\Omega:=\omega(G)+\omega_{FS}(G).	
\end{align*}

Let $G_0$ be a strongly pseudoconvex complex Finsler metric on $E^*$. We consider the following flow: 
\begin{align}\label{flow}
\begin{cases}
&\frac{\p}{\p t}\log G=\Delta_{\Omega}\log G,\\
& \omega_{FS}(G)>0,\\
&G(0)=G_0.	
\end{cases}
\end{align}
Here $\Delta_{\Omega}:=\sqrt{-1}\Lambda \p\b{\p}$, $\Lambda$ is the adjoint operator of $\Omega\wedge \bullet$. 

Note that
\begin{align}\label{4.1}
\Delta_{\Omega}\log G=\Lambda\sqrt{-1}\p\b{\p}\log G=\Lambda_{\omega(G)}(-\Psi)+(r-1)	
\end{align}
is a smooth function on $P(E^*)$, then 
\begin{align*}
G(t)=e^{\int^t_0\Delta_{\Omega}\log G dt}G_0, 	
\end{align*}
By Definition \ref{defn1}, one sees that  $G(t)$ is always a complex Finsler metric on $E^*$. Moreover,  $G_0$ is strongly pseudoconvex, i.e. $\omega_{FS}(G_0)>0$, so $G(t)$ is strongly pseudoconvex automatically as $t$ small.

\begin{rem}\label{rem2}
\begin{enumerate}
  \item For the case 
  $$\omega(G)=\sqrt{-1}g_{\alpha\b{\beta}}dz^{\alpha}\wedge d\b{z}^{\beta}$$
  is a fixed K\"ahler metric on $M$, and $G_0=h_0^{i\b{j}}v_i\b{v}_j$ comes from a Hermitian metric $(h_0^{i\b{j}})$ of $E^*$,  then   
  \begin{align}\label{3.1}
  \begin{split}
  	0&=\frac{\p}{\p t}\log G-\Delta_{\Omega}\log G\\
  	&=\frac{1}{G}\frac{\p G}{\p t}+\Lambda_{\omega}\Psi-(r-1)\\
  	&=\frac{v_i\o{v_j}}{G}\left(\frac{\p h^{i\b{j}}(t)}{\p t}+g^{\alpha\b{\beta}}R^{i\b{j}}_{~~\alpha\b{\beta}}-(r-1)h^{\b{j}i}(t)\right),
  \end{split}	
  \end{align}
where $h^{i\b{j}}(t):=\frac{\p^2 G(t)}{\p v_i\p\b{v}_j}$. From above equation, one has
\begin{align*}
	\frac{\p h^{i\b{j}}(t)}{\p t}+g^{\alpha\b{\beta}}\frac{\p^2}{\p v_i\p\b{v}_j}(R^{i\b{j}}_{~~\alpha\b{\beta}}v_i\b{v}_j)-(r-1)h^{\b{j}i}(t)=0.
\end{align*}
Since $h^{i\b{j}}(0)=h_0^{i\b{j}}$ is a Hermitian metric, which is independent of the vertical coordinates  $\{v_i, 1\leq i\leq r\}$, so $h^{i\b{j}}(t)$ is also a Hermitian metric. In fact, by induction, we assume that $(\frac{\p^k h^{i\b{j}}}{\p t^k})|_{t=0}$ is independent of fibers, then  
\begin{align*}
\begin{split}
	\frac{\p^{k+1} h^{i\b{j}}(t)}{\p t^{k+1}}|_{t=0}&=\frac{\p^k}{\p t^k}\left(-g^{\alpha\b{\beta}}\frac{\p^2}{\p v_i\p\b{v}_j}(R^{i\b{j}}_{~~\alpha\b{\beta}}v_i\b{v}_j)+(r-1)h^{\b{j}i}(t)\right)|_{t=0}\\
	&=-g^{\alpha\b{\beta}}\frac{\p^2}{\p v_i\p\b{v}_j}\left((\frac{\p^k}{\p t^k}R^{i\b{j}}_{~~\alpha\b{\beta}})|_{t=0}v_i\b{v}_j\right)+(r-1)(\frac{\p^k}{\p t^k}h^{\b{j}i})|_{t=0}\\
	&=-g^{\alpha\b{\beta}}(\frac{\p^k}{\p t^k}R^{i\b{j}}_{~~\alpha\b{\beta}})|_{t=0}+(r-1)(\frac{\p^k}{\p t^k}h^{\b{j}i})|_{t=0},
	\end{split}
\end{align*}
 which is also independent of fibers, because $\frac{\p^k}{\p t^k}R^{i\b{j}}_{~~\alpha\b{\beta}}$ is the combination by $h^{\b{j}i}$ and $\frac{\p^l}{\p t^l}h^{\b{j}i}, 1\leq l\leq k$. It follows that 
 \begin{align*}
 h^{\b{j}i}(t)=h^{\b{j}i}(0)+\frac{\p h^{i\b{j}}(t)}{\p t}|_{t=0}t+\cdots+\frac{\p^{k} h^{i\b{j}}(t)}{\p t^{k}}|_{t=0}t^k+\cdots	
 \end{align*}
is independent of fibers for small $t$. By Proposition \ref{prop1} and $\omega_{FS}(G)>0$, $h^{\b{j}i}(t)$ is a positive definite matrix, so $(h^{\b{j}i}(t))$ is a Hermitian metric on $E^*$. 

Therefore, (\ref{3.1}) is equivalent to 
\begin{align}\label{3.2}
\frac{\p h^{i\b{j}}(t)}{\p t}+g^{\alpha\b{\beta}}(R^{h^{-1}})^{i\b{j}}_{~~\alpha\b{\beta}}-(r-1)h^{\b{j}i}(t)=0.	
\end{align}
Multiplying by $h_{k\b{j}}$ to both sides of (\ref{3.2}), one has
\begin{align*}
h^{-1}\cdot \frac{\p h}{\p t}+\Lambda R^h+(r-1)I:=h^{\b{j}i}\frac{\p h_{k\b{j}}}{\p t}+g^{\alpha\b{\beta}}(R^h)^i_{k\alpha\b{\beta}}+(r-1)\delta_{k}^i=0,
\end{align*}
which is exactly the Hermitian-Yang-Mills flow \cite{Atiyah, Don1} (see also \cite{Ko3}).

\item For the case of $E=TM$. Let 
\begin{align*}
\omega(G)=\sqrt{-1}g_{\alpha\b{\beta}}dz^{\alpha}\wedge d\b{z}^{\beta},
\end{align*}
where $(g_{\alpha\b{\beta}})$ denotes the inverse of the matrix  $\left(\frac{\p^2 G}{\p v_{\alpha}\p\b{v}_{\beta}}\right)$. Let $G_0=g_0^{\alpha\b{\beta}}v_\alpha\b{v}_\beta$ be a strongly pseudoconvex complex Finsler metric on $T^*M$ induced by the Hermitian metric 
\begin{align*}
\omega_0=\sqrt{-1}(g_0)_{\alpha\b{\beta}}dz^{\alpha}\wedge d\b{z}^{\beta}.
\end{align*}
Similar to the above case, the flow (\ref{flow}) is equivalent to the following {\it Hermitian curvature flow}
\begin{align}\label{KR}
	\frac{\p g_{\gamma\b{\delta}}}{\p t}+g^{\alpha\b{\beta}}R_{\gamma\b{\delta}\alpha\b{\beta}}+(r-1)g_{\gamma\b{\delta}}=0,
\end{align}
which was first given in \cite[(3)]{Streets}, and many results were obtained there on this kind of flows with an arbitrary quadratic tensor $Q$ in the torsion. This flow (\ref{KR})
was also appeared in \cite[(7.11)]{Liu2}, and by \cite[Theorem 7.1]{Liu2} (or \cite[Remark 7.2]{Liu2}), if the initial metric $\omega_0$ is a K\"ahler metric, then this flow is reduced to the usual K\"ahler-Ricci flow (see \cite{Cao}).
\end{enumerate}
\end{rem}

\subsection{Positivity preserving along the flow}

In this subsection, we shall discuss the positivity preserving along the flow (\ref{flow}). We assume that the initial metric $G(0)=G_0$ satisfies
$\sqrt{-1}\p\b{\p}\log G_0\geq 0$, and   will prove 
\begin{prop}\label{Mainprop}
Let $(z_0,[v_0],t_0)$ be a point and time such that $\sqrt{-1}\p\b{\p}\log G\geq 0$ for all $0\leq t<t_0,$
and there is a $(1,0)$-type vector $u$ such that $i_u(\p\b{\p}\log G)(z_0,[v_0],t_0)=0$. Then 
	\begin{align*}
	\frac{\p}{\p t}\left(\p\b{\p}\log G(U,\o{U})\right)=\Delta^H_{\Omega}(\p\b{\p}\log G(U,\o{U}))+\langle R^g(u,\o{u}), -\Psi\rangle_{\Omega}-\left|i_u\p^V\Psi\right|^2_{\Omega}.
\end{align*}
at this point $(z_0, [v_0], t_0)$, where $U$ is defined by (\ref{3.111}), which is a locally extended vector field of $u$, and   
$$\langle \sqrt{-1}R^g(u,\o{u}), -\Psi\rangle_{\Omega}:=(-\Psi)_{\alpha\b{\delta}}g^{\alpha\b{\beta}}g^{\gamma\b{\delta}}R^g_{\gamma\b{\beta}\sigma\b{\tau}}u^\sigma\b{u}^\tau,\quad \left|i_u\p^V\Psi\right|^2_{\Omega}:=(\log G)^{a\b{b}}\p_a\Psi_{\alpha\b{\beta}}\o{\p_b\Psi_{\gamma\b{\tau}}}u^\alpha\o{u}^{\gamma}g^{\tau\b{\beta}}.$$
Here $\Omega=\omega(G)+\omega_{FS}(G)$, $\omega(G)=p^*\omega$, $\omega$ is a K\"ahler metric on $M$  depending on the Finsler metric $G$. 
\end{prop}
\vspace{3mm}
Firstly, let $\epsilon>0$ small enough such that 
\begin{align}
\Omega_{\epsilon}=\omega(G)+\epsilon\sqrt{-1}\p\b{\p}\log G=\sqrt{-1}(g_{\alpha\b{\beta}}-\epsilon\Psi_{\alpha\b{\beta}})dz^{\alpha}\wedge d\b{z}^{\beta}+\epsilon\sqrt{-1}\frac{\p^2\log G}{\p w^a\p\b{w}^b}\delta w^a\wedge \delta\b{w}^b	
\end{align}
is a Hermitian metric on $P(E^*)$, where $\Psi_{\alpha\b{\beta}}$ is given by $\Psi=\sqrt{-1}\Psi_{\alpha\b{\beta}}dz^\alpha\wedge d\b{z}^{\beta}$. Denote by $\n^{\epsilon}$ the Chern connection of the Hermitian metric $\Omega_{\epsilon}$, which is the unique connection preserving the holomorphic structure and the metric $\Omega_{\e}$. 
For any two $(1,0)$-type vector fields $X, Y$ of $P(E^*)$, then
\begin{align}\label{con}
\begin{split}
\n^{\epsilon}_X Y &=\langle \n^{\epsilon}_X Y, \frac{\delta}{\delta z^{\beta}}\rangle_{\epsilon} g^{\b{\beta}\alpha}_{\epsilon}\frac{\delta}{\delta z^{\alpha}}+\langle \n^{\epsilon}_X Y, \frac{\p}{\p w^{b}}\rangle_{\epsilon} \frac{1}{\e}(\log G)^{\b{b}a}\frac{\p}{\p w^{a}}\\
&=	X\langle Y, \frac{\delta}{\delta z^{\beta}}\rangle_{\epsilon} g^{\b{\beta}\alpha}_{\epsilon}\frac{\delta}{\delta z^{\alpha}}-\langle Y, \o{X}(\frac{\delta}{\delta z^{\beta}})\rangle_{\epsilon} g^{\b{\beta}\alpha}_{\epsilon}\frac{\delta}{\delta z^{\alpha}}+X\langle Y^V, \frac{\p}{\p w^{b}}\rangle_{\epsilon}  \frac{1}{\e}(\log G)^{\b{b}a}\frac{\p}{\p w^{a}}\\
&=	X\langle Y, \frac{\delta}{\delta z^{\beta}}\rangle_{\epsilon} g^{\b{\beta}\alpha}_{\epsilon}\frac{\delta}{\delta z^{\alpha}}+\langle Y^V, \o{X}(N^b_{\beta})\frac{\p}{\p w^b}\rangle_{\epsilon} g^{\b{\beta}\alpha}_{\epsilon}\frac{\delta}{\delta z^{\alpha}}+X\langle Y^V, \frac{\p}{\p w^{b}}\rangle_{\epsilon}  \frac{1}{\e}(\log G)^{\b{b}a}\frac{\p}{\p w^{a}},
\end{split}
\end{align}
where $Y^V$ denotes the vertical part of $Y$, $(g^{\b{\beta}\alpha})_{\epsilon}$ denotes the inverse of $g_{\epsilon\alpha\b{\beta}}:=g_{\alpha\b{\beta}}-\epsilon\Psi_{\alpha\b{\beta}}$, $\langle\cdot,\cdot\rangle_{\epsilon}$ is the inner product defined by $\Omega_{\epsilon}$. 

 Denote by $Z^A$ the coordinates $z^\alpha$ or $w^a$, $1\leq A\leq n+r-1$. We rewrite $\Omega_{\epsilon}$ as the following form:
\begin{align}
\Omega_{\epsilon}=\sqrt{-1}\Omega_{\epsilon A\b{B}}dZ^A\wedge d\b{Z}^B.	
\end{align}
In this form, the Chern connection is given by
\begin{align}
\n^{\epsilon}_{\frac{\p}{\p Z^A}}\frac{\p}{\p Z^B}=\Gamma^C_{AB}\frac{\p}{\p Z^C},\quad \Gamma^C_{AB}=\frac{\p \Omega_{\epsilon B\b{D}}}{\p Z^A}\Omega^{\b{D}C}_{\epsilon}.	
\end{align}
Here $(\Omega^{\b{D}C}_{\epsilon})$ denotes the inverse of  the matrix $(\Omega_{\epsilon C\b{D}})$.
The Chern curvature tensor of $\Omega_{\epsilon}$ is defined by 
\begin{align}\label{cur}
\begin{split}
R^{\epsilon}_{A\b{B}C\b{D}}&:=R(\frac{\p}{\p Z^A},\frac{\p}{\p \b{Z}^B},\frac{\p}{\p Z^C},\frac{\p}{\p \b{Z}^D})\\
&:=\langle\frac{\p}{\p Z^A},(\n^{\epsilon}_{\frac{\p}{\p Z^D}}\n^{\epsilon}_{\frac{\p}{\p \b{Z}^C}}-\n^{\epsilon}_{\frac{\p}{\p \b{Z}^C}}\n^{\epsilon}_{\frac{\p}{\p Z^D}}-\n^{\epsilon}_{[\frac{\p}{\p Z^D},\frac{\p}{\p \b{Z}^C}]})\frac{\p}{\p Z^B}\rangle_{\epsilon}\\
&=-\Omega_{\epsilon A\b{E}}\frac{\p\o{\Gamma^E_{DB}}}{\p Z^C}.
\end{split}
\end{align}
The Chern connection $\n^{\epsilon}$ induces a natural connection on the cotangent bundle $T^*P(E^*)$ ( resp. $\o{T^*P(E^*)}$ ) by
\begin{align}
	\n^{\epsilon}_{A}(f_C dZ^C)=(\p_A f_C-\Gamma^B_{AC}f_{B})dZ^C,\quad (\text{resp.} \,\, \n^{\epsilon}_{\b{B}}(f_{\b{D}} d\b{Z}^D)=(\p_{\b{B}} f_{\b{D}}-\o{\Gamma^E_{BD}}f_{\b{E}})dZ^C)
\end{align}
  for any smooth $(1,0)$-form $f_C dZ^C$ (resp. $(0,1)$-form $f_{\b{D}}d\b{Z}^D$), where $\n^{\epsilon}_A:=\n^{\epsilon}_{\frac{\p}{\p Z^A}}$.
   For convenience, we denote 
   \begin{align}
   \n^{\epsilon}_A f_C:=\p_A f_C-\Gamma^B_{AC}f_{B},\quad \n^{\epsilon}_{\b{B}}f_{\b{D}}:=\p_{\b{B}} f_{\b{D}}-\o{\Gamma^E_{BD}}f_{\b{E}}
   \end{align}
and 
\begin{align}\label{3.3}
(\n^{\epsilon}_A f_{C\b{D}})dZ^C\wedge d\b{Z}^D:=\n^{\epsilon}_{\frac{\p}{\p Z^A}}(f_{C\b{D}}dZ^C\wedge d\b{Z}^D),\quad \n^{\epsilon}_A f_{C\b{D}}=\p_A f_{C\b{D}}-\Gamma^B_{AC}f_{B\b{D}}.
\end{align}
By using the above notations, we have 
\begin{align}\label{3.4}
\n^{\epsilon}_A \Omega_{\epsilon C\b{D}}=\p_A\Omega_{\epsilon C\b{D}}-\Gamma^B_{AC}\Omega_{\epsilon B\b{D}}=0.	
\end{align}

By taking $\p\b{\p}$ to the both sides of the first equation of (\ref{flow}), one has
\begin{align}
	\frac{\p}{\p t}\p\b{\p}\log G=\p\b{\p}\Delta_{\Omega}\log G=\p\b{\p}(tr_{\omega(G)}(-\Psi)),
\end{align}
where the last equality follows from (\ref{4.1}). Since $\lim_{\epsilon\to 0}(\omega(G)-\epsilon\Psi)=\omega(G)$, so
\begin{align}\label{flow1}
\begin{split}
\frac{\p}{\p t}\p\b{\p}\log G &=\lim_{\epsilon\to 0}\p\b{\p}(tr_{\omega(G)-\epsilon\Psi}(-\Psi))\\
&=\lim_{\epsilon\to 0}\p\b{\p}\Delta_{\Omega_{\epsilon}}\log G\\
&=\lim_{\epsilon\to 0}\p_C\p_{\b{D}}(\Omega^{A\b{B}}_{\epsilon}\p_A\p_{\b{B}}\log G)dZ^C\wedge d\b{Z}^D.
\end{split}	
\end{align}

Denote $f:=\log G$ and 
$$f_{A\b{B}}:=\p_A\p_{\b{B}}f,\quad f_{A\b{B}C}:=\p_A\p_{\b{B}}\p_C f,\quad f_{A\b{B}C\b{D}}:=\p_A\p_{\b{B}}\p_C\p_{\b{D}} f,\quad etc..$$ 
By (\ref{3.3}) and (\ref{3.4}), one has
\begin{align}\label{3.5}
\begin{split}
\p_C\p_{\b{D}}(\Omega_{\epsilon}^{A\b{B}}\p_A\p_{\b{B}}\log G)&=\p_C\p_{\b{D}}(\Omega_{\epsilon}^{A\b{B}}f_{A\b{B}})=
	\n^{\epsilon}_C\n^{\epsilon}_{\b{D}}(\Omega_{\epsilon}^{AB}f_{A\b{B}})\\
	&=\Omega^{A\b{B}}_{\epsilon}\n^{\epsilon}_C\n^{\epsilon}_{\b{D}}f_{A\b{B}}\\
	&=\Omega_{\epsilon}^{A\b{B}}\n^{\epsilon}_C(f_{A\b{B}\b{D}}-\o{\Gamma^E_{DB}}f_{A\b{E}})\\
	&=\Omega_{\epsilon}^{A\b{B}}(f_{A\b{B}C\b{D}}-\Gamma^{F}_{CA}f_{F\b{B}\b{D}}-\p_C\o{\Gamma^E_{DB}}f_{A\o{E}}\\
	&\quad -\o{\Gamma^E_{DB}}(f_{AC\b{E}}-\Gamma^F_{CA}f_{F\b{E}})).	
	\end{split}
\end{align}
Similarly, 
\begin{align}\label{3.6}
\begin{split}
&\quad \Omega_{\epsilon}^{A\b{B}}(\n^{\epsilon}_A\n^{\epsilon}_{\b{B}}f_{C\b{D}})\\
&=\Omega^{A\b{B}}_{\epsilon}(f_{A\b{B}C\b{D}}-\Gamma^{F}_{AC}f_{F\b{B}\b{D}}-\p_A\o{\Gamma^E_{BD}}f_{C\o{E}}-\o{\Gamma^E_{BD}}(f_{AC\b{E}}-\Gamma^F_{AC}f_{F\b{E}})).	
\end{split}
\end{align}
Combining (\ref{3.5}) with (\ref{3.6}), we obtain
\begin{align}\label{3.7}
\begin{split}
	&\quad \p_C\p_{\b{D}}(\Omega_{\epsilon}^{A\b{B}}f_{A\b{B}})-\Omega_{\epsilon}^{A\b{B}}(\n^{\epsilon}_A\n^{\epsilon}_{\b{B}}f_{C\b{D}})\\
	&=\Omega_{\epsilon}^{A\b{B}}(\Gamma^F_{AC}-\Gamma^F_{CA})f_{F\b{D}\b{B}}+\Omega_{\epsilon}^{A\b{B}}(\o{\Gamma^E_{BD}}-\o{\Gamma^E_{DB}})f_{AC\b{E}}+\Omega_{\epsilon}^{A\b{B}}(\o{\Gamma^E_{DB}}\Gamma^F_{CA}-\o{\Gamma^E_{BD}}\Gamma^F_{AC})f_{F\b{E}}\\
	&\quad+\Omega_{\epsilon}^{A\b{B}}\p_A\o{\Gamma^E_{BD}}f_{C\b{E}}-\Omega_{\epsilon}^{A\b{B}}\p_C\o{\Gamma^E_{DB}}f_{A\b{E}}\\
		&=\Omega_{\epsilon}^{A\b{B}}(\Gamma^F_{AC}-\Gamma^F_{CA})\n_{\b{D}}f_{F\b{B}}+\Omega_{\epsilon}^{A\b{B}}(\o{\Gamma^E_{BD}}-\o{\Gamma^{E}_{DB}})\n_A f_{C\b{E}}\\
	&\quad-\Omega_{\epsilon}^{A\b{B}}\Omega_{\epsilon}^{F\b{E}}R^{\epsilon}_{F\b{D}A\b{B}}f_{C\b{E}}+\Omega_{\epsilon}^{A\b{B}}\Omega_{\epsilon}^{F\b{E}}R^{\epsilon}_{F\b{B}C\b{D}}f_{A\b{E}}.
	\end{split}
	\end{align}

Substituting (\ref{3.7}) into (\ref{flow1}), we have
\begin{align}\label{2.1}
\begin{split}
	&\quad \frac{\p}{\p t}\p\b{\p}\log G=\lim_{\epsilon\to 0}\left(\Omega_{\epsilon}^{A\b{B}}\n^{\epsilon}_A\n^{\epsilon}_{\b{B}}(\p\b{\p}\log G)\right.\\
	&+\left(\Omega_{\epsilon}^{A\b{B}}(\Gamma^F_{AC}-\Gamma^F_{CA})\n_{\b{D}}f_{F\b{B}}+\Omega_{\epsilon}^{A\b{B}}(\o{\Gamma^E_{BD}}-\o{\Gamma^{E}_{DB}})\n_A f_{C\b{E}}\right.\\
	&\left.\left.\quad-\Omega_{\epsilon}^{A\b{B}}\Omega_{\epsilon}^{F\b{E}}R^{\epsilon}_{F\b{D}A\b{B}}f_{C\b{E}}+\Omega_{\epsilon}^{A\b{B}}\Omega_{\epsilon}^{F\b{E}}R^{\epsilon}_{F\b{B}C\b{D}}f_{A\b{E}}\right) dZ^C\wedge d\b{Z}^D\right).
	\end{split}
\end{align}

As in the proof of Theorem \ref{thm1}, we assume that $\sqrt{-1}\p\b{\p}\log G\geq 0$ for all $0\leq t<t_0$, and $(z_0, [v_0], t_0)$ is a point and time, and $u=u^{A}\frac{\p}{\p Z^{A}}$ is a vector such that 
	\begin{align}\label{3.8}
	f_{C\b{D}}u^C(z_0,[v_0],t_0)=(\p_C\p_{\b{D}}\log G)u^{C}(z_0,[v_0],t_0)=0
	\end{align}
and
\begin{align}\label{assume1}
	(\p\b{\p}\log G(W,\o{W}))(z,[v],t)=(\p_C\p_{\b{D}}\log G)W^{C}\o{W^D}(z,[v],t)\geq 0
\end{align}
for all $(z,[v])\in P(E^*)$, $t\in [0,t_0]$, and tangent vectors $W\in T_{(z,[v])} P(E^*)$. This implies that 
\begin{align}\label{3.10}
u=u^\alpha\frac{\delta}{\delta z^{\alpha}}\in q_*\mc{H}.	
\end{align}
Indeed, one may assume that $u=u_1+u_2$, where $u_1=u^\alpha\frac{\delta}{\delta z^{\alpha}}$, $u_2=u^a\frac{\p}{\p w^{a}}$ are the horizontal and vertical parts of $u$ respectively. By (\ref{1.14}) and $\omega_{FS}(G)>0$, one has
\begin{align}
\begin{split}
0&=(\sqrt{-1}\p\b{\p}\log G)(u,\b{u})\\
&=(\sqrt{-1}\p\b{\p}\log G)(u_1+u_2,\o{u_1}+\o{u_2})\\
&=(\sqrt{-1}\p\b{\p}\log G)(u_1,\o{u_1})+\omega_{FS}(G)(u_2,\o{u_2})\\
&\geq \omega_{FS}(G)(u_2,\o{u_2})\geq 0,
\end{split}
\end{align}
and all equalities hold if and only if $u_2=0$, namely $u=u_1$.
From Lemma \ref{lemma1} and (\ref{3.10}), (\ref{3.8}) is equivalent to 
\begin{align}\label{3.30}
(i_{u}\Psi)(z_0,[v_0], t_0)=0.
\end{align}

For any $\epsilon>0$, by parallel translation, one can extend $u$ to a vector field $U_{\epsilon}=U^A_{\epsilon}\frac{\p}{\p Z^A}$  defined in a neighborhood of $(z_0,[v_0],t_0)$ such that $U_{\epsilon}(z_0,[v_0],t_0)=u$ and 
\begin{align}\label{3.9}
\frac{\p U_{\epsilon}}{\p t}(z_0,[v_0],t_0)=0,\quad (\n^{\epsilon} U_{\epsilon})(z_0,[v_0],t_0)=0.	
\end{align}
This can be done by parallel translating $u$ along radial rays with respect to the connection $\n^{\epsilon}$, and then by extending to be independent of time $t$.

We assume that 
$$U_{\epsilon}(z,[v],t_0)=U_{\epsilon}^{\alpha}\frac{\delta}{\delta z^{\alpha}}+U_{\epsilon}^a\frac{\p}{\p w^a}.$$
 By (\ref{con}), one has 
\begin{align}
\n^{\epsilon} U_{\epsilon}= \b{\p}U_{\epsilon}+\left(\p(U_{\epsilon}^{\alpha}g_{\epsilon\alpha\b{\beta}})+U_{\epsilon}^a\p(\o{N^b_{\beta}})\epsilon(\log G)_{a\b{b}}\right)g^{\b{\beta}\gamma}_{\epsilon}\frac{\delta}{\delta z^{\gamma}}+\p(U_{\epsilon}^a(\log G)_{a\b{b}})(\log G)^{\b{b}c}\frac{\p}{\p w^c}.	
\end{align}
So the second equation of (\ref{3.9}) is equivalent to 
\begin{align}\label{4.3}
\begin{cases}
&\b{\p}U^{\alpha}_{\epsilon}=0,\\
&\b{\p}(-U_{\epsilon}^\alpha N^a_{\alpha}+U^{a}_{\epsilon})=0,\\
&\p(U_{\epsilon}^{\alpha}g_{\epsilon\alpha\b{\beta}})+U_{\epsilon}^a\p(\o{N^b_{\beta}})\epsilon(\log G)_{a\b{b}}=0,\\
&\p(U_{\epsilon}^a(\log G)_{a\b{b}})(\log G)^{\b{b}c}=0,\\
\end{cases}
\end{align}
at the point $(z_0,[v_0],t_0)$. By (\ref{3.10}), $U_{\epsilon}^a=0$ at the point $(z_0,[v_0],t_0)$, so (\ref{4.3}) is equivalent to
\begin{align}\label{3.11}
\begin{cases}
&\b{\p}U_{\epsilon}^{\alpha}=0,\\
&\b{\p}U_{\epsilon}^a=u^{\alpha}\b{\p}N^a_{\alpha},\\
&\p U_{\epsilon}^{\alpha}+g_{\epsilon}^{\b{\beta}\alpha}\p g_{\epsilon\gamma\b{\beta}}u^{\gamma}=0,\\
&\p U_{\epsilon}^a=0.\\
\end{cases}
\end{align}
Since $\lim_{\epsilon\to 0}g_{\epsilon \alpha\b{\beta}}=g_{\alpha\b{\beta}}$, so $\lim_{\epsilon\to 0} U_{\epsilon}=U$ which satisfies the following equations:
\begin{align}\label{3.111}
\begin{cases}
&\b{\p}U^{\alpha}=0,\\
&\b{\p}U^a=u^{\alpha}\b{\p}N^a_{\alpha},\\
&\p U^{\alpha}+g^{\b{\beta}\alpha}\p g_{\gamma\b{\beta}}u^{\gamma}=0,\\
&\p U^a=0
\end{cases}
\end{align}
and $\frac{\p U}{\p t}=0$ at the point $(z_0,[v_0],t_0)$.
By (\ref{2.1}), (\ref{3.8}) and (\ref{3.9}), one has at the point $(z_0,[v_0],t_0)$,
\begin{align}\label{4.2}
\begin{split}
	&\frac{\p}{\p t}\left(\p\b{\p}\log G(U,\o{U})\right)=\lim_{\epsilon\to 0}\left(\Omega_{\epsilon}^{A\b{B}}\n^{\epsilon}_A\n^{\epsilon}_{\b{B}}(\p\b{\p}\log G)(u,\o{u})\right.\\
	&+\left(\Omega_{\epsilon}^{A\b{B}}(\Gamma^F_{AC}-\Gamma^F_{CA})\n_{\b{D}}f_{F\b{B}}+\Omega_{\epsilon}^{A\b{B}}(\o{\Gamma^E_{BD}}-\o{\Gamma^{E}_{DB}})\n_A f_{C\b{E}}\right.\\
	&\left.\left.\quad-\Omega_{\epsilon}^{A\b{B}}\Omega_{\epsilon}^{F\b{E}}R_{F\b{D}A\b{B}}f_{C\b{E}}+\Omega_{\epsilon}^{A\b{B}}\Omega_{\epsilon}^{F\b{E}}R^{\epsilon}_{F\b{B}C\b{D}}f_{A\b{E}}\right) u^C\b{u}^D\right)\\
	&=\lim_{\epsilon\to 0}\left(\Omega_{\epsilon}^{A\b{B}}\p_A\p_{\b{B}}(\p\b{\p}\log G(U_{\epsilon},\o{U}_{\epsilon}))\right.\\
	&+\left(\Omega_{\epsilon}^{A\b{B}}(\Gamma^F_{AC}-\Gamma^F_{CA})\n_{\b{D}}f_{F\b{B}}+\Omega_{\epsilon}^{A\b{B}}(\o{\Gamma^E_{BD}}-\o{\Gamma^{E}_{DB}})\n_A f_{C\b{E}}\right.\\
	&\left.\left.\quad+\Omega_{\epsilon}^{A\b{B}}\Omega_{\epsilon}^{F\b{E}}R^{\epsilon}_{F\b{B}C\b{D}}f_{A\b{E}}\right) u^C\b{u}^D\right).
\end{split}	
\end{align}

In order to deal with (\ref{4.2}), we assume that $\omega(G)=p^*\omega$ for some K\"ahler metric $\omega$ on $M$, so
\begin{align}\label{assume}
\Omega_{\epsilon}=\omega(G)+\epsilon\sqrt{-1}\p\b{\p}\log G	
\end{align}
is a K\"ahler metric on $P(E^*)$ for $\epsilon>0$ small enough. Thus, (\ref{4.2}) is reduced to 
\begin{align}\label{3.12}
	\frac{\p}{\p t}\left(\p\b{\p}\log G(U,\o{U})\right)=\lim_{\epsilon\to 0}\left(\Omega_{\epsilon}^{A\b{B}}\p_A\p_{\b{B}}(\p\b{\p}\log G(U_{\epsilon},\o{U}_{\epsilon}))+\Omega_{\epsilon}^{A\b{B}}\Omega_{\epsilon}^{F\b{E}}R^{\epsilon}_{F\b{B}C\b{D}}f_{A\b{E}}u^C\b{u}^D\right).
\end{align}

For the first term in the RHS of (\ref{3.12}), we have 
\begin{align}\label{3.15}
\begin{split}
	\Omega_{\epsilon}^{A\b{B}}\p_A\p_{\b{B}}(\p\b{\p}\log G(U_{\epsilon},\o{U}_{\epsilon})) &=\sqrt{-1}\Lambda \p\b{\p}(\p\b{\p}\log G(U_{\epsilon},\o{U}_{\epsilon}))\\
	&=\Delta^H_{\Omega_{\epsilon}}(\p\b{\p}\log G(U_{\epsilon},\o{U}_{\epsilon}))+\Delta^V_{\Omega_{\epsilon}}(\p\b{\p}\log G(U_{\epsilon},\o{U}_{\epsilon})).
	\end{split}
\end{align}
Here $\Delta^V_{\Omega_{\epsilon}}=\frac{1}{\epsilon}(\log G)^{\b{b}a}\frac{\p^2}{\p w^a\p\b{w}^b}$ is the vertical Laplacian, while
$
	\Delta^H_{\Omega_{\epsilon}} f=g_{\epsilon}^{\alpha\b{\beta}}(\p\b{\p}f)(\frac{\delta}{\delta z^{\alpha}},\frac{\delta}{\delta \b{z}^{\beta}})
$
 is the horizontal Laplacian (see Remark \ref{rem1}). Since $\lim_{\epsilon\to o}g_{\epsilon\alpha\b{\beta}}=g_{\alpha\b{\beta}}$ and $\lim_{\epsilon\to 0}U_{\epsilon}=U$, so 
 \begin{align}\label{2.9}
 \lim_{\epsilon\to 0}\Delta^H_{\Omega_{\epsilon}}(\p\b{\p}\log G(U_{\epsilon},\o{U}_{\epsilon}))=\Delta^H_{\Omega}(\p\b{\p}\log G(U,\o{U})).
 \end{align}
By Remark \ref{rem1} (1), one has
 \begin{align}\label{2.2}
 \begin{split}
 	\Delta^V_{\Omega_{\epsilon}}(\p\b{\p}\log G(U_{\epsilon},\o{U}_{\epsilon}))&=\Delta^V_{\Omega_{\epsilon}}(f_{c\b{d}}U_{\epsilon}^c\o{U}_{\epsilon}^d)+\Delta^V_{\Omega_{\epsilon}}((-\Psi)_{\alpha\b{\beta}}U_{\epsilon}^{\alpha}\o{U}^{\beta}_{\epsilon})\\
 	&=\frac{1}{\epsilon}f^{a\b{b}}\frac{\p^2}{\p w^a\p\b{w}^b}(f_{c\b{d}}U_{\epsilon}^c\o{U}_{\epsilon}^d)+\frac{1}{\epsilon}f^{a\b{b}}\frac{\p^2}{\p w^a\p\b{w}^b}((-\Psi)_{\alpha\b{\beta}}U_{\epsilon}^{\alpha}\o{U}^{\beta}_{\epsilon}).
 	 \end{split}	
 \end{align}
 
 For the first term in the RHS of (\ref{2.2}), by (\ref{3.11}) and $U^c=0$ at the point $(z_0, [v_0], t_0)$, we have  
 \begin{align}\label{2.6}
\frac{1}{\epsilon}f^{a\b{b}}\frac{\p^2}{\p w^a\p\b{w}^b}(f_{c\b{d}}U_{\epsilon}^c\o{U}_{\epsilon}^d)=\frac{1}{\epsilon}f^{a\b{b}}f_{c\b{d}}\p_{\b{b}}U^c_{\epsilon}\p_a \o{U}^d_{\epsilon}=\frac{1}{\epsilon}f^{a\b{b}}f_{c\b{d}} u^{\alpha}\o{u}^{\beta}\p_{\b{b}}N^c_{\alpha}\p_a \o{N}^d_{\beta}.
 \end{align}

The following lemma is actually proved in \cite[(3.46)]{Wan1}. For readers' convenience, we give a proof here.
\begin{lemma}\label{lemma3}
	\begin{align}
	f^{\b{b}a}\frac{\p^2}{\p w^a\p\b{w}^b}(-\Psi)_{\alpha\b{\beta}}=\p\b{\p}\log\det(f_{a\b{b}})(\frac{\delta}{\delta z^{\alpha}},\frac{\delta}{\delta \b{z}^{\beta}})-\langle\b{\p}^V\frac{\delta}{\delta z^{\alpha}}, \b{\p}^V\frac{\delta}{\delta z^{\beta}}\rangle,	
	\end{align}
	where $\b{\p}^V\frac{\delta}{\delta z^{\alpha}}:=\frac{\p}{\p \b{w}^d}(-f_{\alpha\b{b}}f^{\b{b}c})\delta \b{w}^d\otimes \frac{\p}{\p w^c}$ and $\langle\b{\p}^V\frac{\delta}{\delta z^{\alpha}}, \b{\p}^V\frac{\delta}{\delta z^{\beta}}\rangle:=f^{\b{b}a}\p_{\b{b}}N_{\alpha}^c\p_a\o{N}^d_{\beta}f_{c\b{d}}$.
\end{lemma}
\begin{proof}
Let $(-\Psi)_{\alpha\b{\beta}}$ denote the coefficient of $-\Psi$, i.e. $-\Psi=\sqrt{-1}(-\Psi)_{\alpha\b{\beta}}dz^{\alpha}\wedge d\b{z}^{\beta}$, then  
\begin{align}\label{2.333}
(-\Psi)_{\alpha\b{\beta}}=f_{\alpha\b{\beta}}-f_{\alpha\b{d}}f^{\b{d}c}f_{c\b{\beta}}.
\end{align}
In fact, by the decomposition (\ref{1.14}), one has
\begin{align*}
(-\Psi)_{\alpha\b{\beta}}&=(-\sqrt{-1})(-\Psi)(\frac{\delta}{\delta z^{\alpha}},\frac{\delta}{\delta \b{z}^{\beta}})\\
&=(\p\b{\p}f)(\frac{\p}{\p z^{\alpha}}-f_{\alpha\b{b}}f^{\b{b}a}\frac{\p}{\p w^a}, \frac{\p}{\p \b{z}^{\beta}}-f_{\b{\beta}a}f^{a\b{b}}\frac{\p}{\p\b{w}^b})\\
&=f_{\alpha\b{\beta}}-f_{\alpha\b{d}}f^{\b{d}c}f_{c\b{\beta}},
\end{align*}
which proves (\ref{2.333}).

For any fixed point $(z,[v])\in P(E^*)|_{z}$, $z\in M$, we take normal coordinates near $(z,[v])$ such that $f_{a\b{b}}(z,[v])=\delta_{ab}$, $f_{a\b{b}c}(z,[v])=0$. Evaluating at $(z,[v])$ we see that
\begin{align*}
\begin{split}
f^{\b{b}a}\frac{\p^2}{\p w^a\p\b{w}^b}(-\Psi)_{\alpha\b{\beta}}&= f^{\b{b}a}\frac{\p^2}{\p w^a\p\b{w}^b}(f_{\alpha\b{\beta}}-f_{\alpha\b{d}}f^{\b{d}c}f_{c\b{\beta}})\\
&=f^{\b{b}a}\left(f_{a\b{b}\alpha\b{\beta}}-f_{\alpha\b{c}a\b{b}}f_{c\b{\beta}}-f_{\alpha\b{c}}f_{c\b{\beta}a\b{b}}+f_{\alpha\b{d}}f_{a\b{b}d\b{c}}f_{c\b{\beta}}-f_{\alpha\b{c}a}f_{c\b{\beta}\b{b}}-f_{\alpha\b{c}\b{b}}f_{c\b{\beta}a}\right)\\
&=f^{\b{b}a}(-\b{\p}(\p f_{a\b{d}}f^{\b{d}c})f_{c\b{b}})(\frac{\delta}{\delta z^{\alpha}},\frac{\delta}{\delta \b{z}^{\beta}})-f^{\b{b}a}\p_{\b{b}}(-f_{\alpha\b{c}}f^{\b{c}d})\p_a(-f_{k\b{\beta}}f^{k\b{l}})f_{d\b{l}}\\
&=\p\b{\p}\log\det(f_{a\b{b}})(\frac{\delta}{\delta z^{\alpha}},\frac{\delta}{\delta \b{z}^{\beta}})-\langle\b{\p}^V\frac{\delta}{\delta z^{\alpha}}, \b{\p}^V\frac{\delta}{\delta z^{\beta}}\rangle.
\end{split}
\end{align*}
which completes the proof. 
\end{proof}
By Lemma \ref{lemma3} and (\ref{3.11}), one has
\begin{align}\label{2.3}
\begin{split}
	&\quad \frac{1}{\epsilon}f^{a\b{b}}\frac{\p^2}{\p w^a\p\b{w}^b}((-\Psi)_{\alpha\b{\beta}}U_{\epsilon}^{\alpha}\o{U}^{\beta}_{\epsilon})\\
	&=\frac{1}{\epsilon}f^{a\b{b}}\left(\p_a\p_{\b{b}}(-\Psi)_{\alpha\b{\beta}}u^{\alpha}\b{u}^{\beta}+\p_a(-\Psi)_{\alpha\b{\beta}}\p_{\b{b}}\o{U}^{\beta}_{\epsilon}u^{\alpha}+\p_{\b{b}}(-\Psi)_{\alpha\b{\beta}}\p_aU^{\alpha}_{\epsilon}\o{u}^{\beta}+\p_aU_{\epsilon}^{\alpha}\o{\p_b U_{\epsilon}^b}(-\Psi)_{\alpha\b{\beta}}\right)\\
	&=\frac{1}{\epsilon}\p\b{\p}\log\det(f_{a\b{b}})(u,\b{u})-\frac{1}{\epsilon}f^{a\b{b}}f_{c\b{d}} u^{\alpha}\o{u}^{\beta}\p_{\b{b}}N^c_{\alpha}\p_a \o{N}^d_{\beta}-2f^{a\b{b}}\p_a\Psi_{\alpha\b{\beta}}\o{\p_b\Psi_{\gamma\b{\tau}}}u^\alpha\o{u}^{\gamma}g^{\tau\b{\beta}}+O(\epsilon),
	\end{split}
\end{align}
where the last equality follows from $\p_a U^{\alpha}_{\epsilon}=-g^{\b{\beta}\alpha}_{\epsilon}\p_a g_{\epsilon\gamma\b{\beta}}u^{\gamma}=\epsilon g^{\b{\beta}\alpha}_{\epsilon}\p_a\Psi_{\gamma\b{\beta}}u^\gamma=O(\epsilon)$.

Substituting (\ref{2.6}) and (\ref{2.3}) into (\ref{2.2}), we obtain
\begin{align}\label{2.8}
	\Delta^V_{\Omega_{\epsilon}}(\p\b{\p}\log G(U_{\epsilon},\o{U}_{\epsilon}))=\frac{1}{\epsilon}\p\b{\p}\log\det(f_{a\b{b}})(u,\b{u})-2\left|i_u\p^V\Psi\right|^2_{\Omega}+O(\epsilon).
\end{align}
Here we denote $\left|i_u\p^V\Psi\right|^2_{\Omega}:=f^{a\b{b}}\p_a\Psi_{\alpha\b{\beta}}\o{\p_b\Psi_{\gamma\b{\tau}}}u^\alpha\o{u}^{\gamma}g^{\tau\b{\beta}}$.

For the second term in the RHS of (\ref{3.12}), we have
\begin{align}\label{3.21}
	\Omega_{\epsilon}^{A\b{B}}\Omega_{\epsilon}^{F\b{E}}R^{\epsilon}_{F\b{B}C\b{D}}f_{A\b{E}}u^C\b{u}^D=(-\Psi)_{\alpha\b{\delta}}g^{\alpha\b{\beta}}_{\epsilon}g^{\gamma\b{\delta}}_{\epsilon}R^{\epsilon}_{\gamma\b{\beta}\sigma\b{\tau}}u^\sigma\b{u}^\tau+\frac{1}{\epsilon^2}f^{a\b{b}}R^{\epsilon}_{a\b{b}\sigma\b{\tau}}u^{\sigma}\b{u}^{\tau},
\end{align}
where $R^{\epsilon}_{\gamma\b{\beta}\sigma\b{\tau}}=R^{\epsilon}(\frac{\delta}{\delta z^\gamma}, \frac{\delta}{\delta \b{z}^\beta},\frac{\delta}{\delta z^\sigma}\frac{\delta}{\delta \b{z}^\tau})$ and $R^{\epsilon}_{a\b{b}\sigma\b{\tau}}=R^{\epsilon}(\frac{\p}{\p w^a}, \frac{\p}{\p\b{w}^b},\frac{\delta}{\delta z^\sigma}, \frac{\delta}{\delta \b{z}^\tau})$.
By (\ref{con}) and (\ref{cur}), one has
\begin{align}\label{3.22}
\begin{split}
R^{\epsilon}_{\gamma\b{\beta}\sigma\b{\tau}}&=R^{\epsilon}(\frac{\delta}{\delta z^\gamma}, \frac{\delta}{\delta \b{z}^\beta},\frac{\delta}{\delta z^\sigma}, \frac{\delta}{\delta \b{z}^\tau})\\
&=\langle \frac{\delta}{\delta z^\gamma},(\n^{\epsilon}_{\frac{\delta}{\delta z^\tau}}\n^{\epsilon}_{\frac{\delta}{\delta \b{z}^\sigma}}-\n^{\epsilon}_{\frac{\delta}{\delta \b{z}^\sigma}}\n^{\epsilon}_{\frac{\delta}{\delta z^\tau}}-\n^{\epsilon}_{[\frac{\delta}{\delta z^\tau},\frac{\delta}{\delta \b{z}^\sigma}]})\frac{\delta}{\delta z^\beta}\rangle_{\epsilon}\\
&=\left(\b{\p}(\p g_{\epsilon\gamma\b{\delta}}\cdot g_{\e}^{\b{\delta}\alpha})g_{\e\alpha\b{\beta}}\right)(\frac{\delta}{\delta z^{\sigma}},\frac{\delta}{\delta\b{z}^{\tau}})-\epsilon f_{c\b{d}}\frac{\delta}{\delta z^{\sigma}}\o{N}^d_{\beta}\frac{\delta}{\delta\b{z}^{\tau}}N^c_{\gamma}\\
&=R^{g}_{\gamma\b{\beta}\sigma\b{\tau}}+O(\e),
\end{split}
\end{align}
where $R^{g}_{\gamma\b{\beta}\sigma\b{\tau}}:=-\frac{\p^2 g_{\gamma\b{\beta}}}{\p z^{\sigma}\p\b{z}^{\tau}}+g^{\b{\tau}\alpha}\frac{\p g_{\alpha\b{\beta}}}{\p\b{z}^{\tau}}\frac{\p g_{\gamma\b{\delta}}}{\p z^{\sigma}}$
denotes the Chern curvature of the K\"ahler metric $\omega=\sqrt{-1}g_{\alpha\b{\beta}}dz^{\alpha}\wedge d\b{z}^{\beta}$. And 
\begin{align}\label{3.23}
\begin{split}
	\frac{1}{\e^2}f^{a\b{b}}R^{\epsilon}_{a\b{b}\sigma\b{\tau}}u^{\sigma}\o{u}^{\tau}&=\frac{1}{\e^2}f^{a\b{b}}R^{\e}(\frac{\p}{\p w^a}, \frac{\p}{\p\b{w}^b},\frac{\delta}{\delta z^\sigma}, \frac{\delta}{\delta \b{z}^\tau})u^{\sigma}\o{u}^{\tau}\\
	&=\frac{1}{\epsilon^2}f^{a\b{b}}\langle \frac{\p}{\p w^a},(\n^{\e}_{\frac{\delta}{\delta z^\tau}}\n^{\e}_{\frac{\delta}{\delta \b{z}^\sigma}}-\n^{\e}_{\frac{\delta}{\delta \b{z}^\sigma}}\n^{\e}_{\frac{\delta}{\delta z^\tau}}-\n^{\e}_{[\frac{\delta}{\delta z^\tau},\frac{\delta}{\delta \b{z}^\sigma}]})\frac{\p}{\p w^b }\rangle_{\epsilon} u^{\sigma}\o{u}^{\tau}\\
	&=\frac{1}{\e}f^{a\b{b}}\left(\b{\p}(\p f_{a\b{d}}\cdot f^{\b{d}c})f_{c\b{b}}\right)(u,\o{u})+\frac{\delta}{\delta z^{\sigma}}\o{N}^{b}_{\beta}\frac{\delta}{\delta\b{z}^{\tau}}N^a_{\alpha}f_{a\b{b}}g^{\b{\beta}\alpha}_{\epsilon}u^{\sigma}\o{u}^{\tau}\\
	&=-\frac{1}{\e}\p\b{\p}\log\det(f_{a\b{b}})(u,\o{u})+\left|i_u\p^V\Psi\right|^2_{\Omega}+O(\epsilon),
\end{split}	
\end{align}
where the last equality follows from the following equalities:
\begin{align*}
\frac{\delta}{\delta\b{z}^{\gamma}}N^a_\alpha f_{a\b{b}}&=(\p_{\b{\gamma}}-f_{c\b{\gamma}}f^{c\b{e}}\p_{\b{e}})(f_{\alpha\b{d}}f^{\b{d}a})f_{a\b{b}}\\
&=	f_{a\b{b}\b{\gamma}}-f_{\alpha\b{d}}f^{\b{d}a}f_{a\b{b}\b{\gamma}}-f_{c\b{\gamma}}f^{c\b{e}}f_{\b{e}\alpha\b{b}}+f_{c\b{\gamma}}f^{c\b{e}}f_{\alpha\b{d}}f^{\b{d}a}f_{a\b{b}\b{e}}\\
&=\p_{\b{b}}(f_{\alpha\b{\gamma}}-f_{\alpha\b{b}}f^{\b{b}a}f_{a\b{\gamma}})\\
&=\p_{\b{b}}(-\Psi)_{\alpha\b{\gamma}}.
\end{align*}
Substituting (\ref{3.22}) and (\ref{3.23}) into (\ref{3.21}), we have 
\begin{align}\label{2.7}
	\Omega_{\epsilon}^{A\b{B}}\Omega_{\epsilon}^{F\b{E}}R^{\epsilon}_{F\b{B}C\b{D}}f_{A\b{E}}u^C\b{u}^D=-\frac{1}{\e}\p\b{\p}\log\det(f_{a\b{b}})(u,\o{u})+\left|i_u\p^V\Psi\right|^2_{\Omega}+\langle R^g(u,\o{u}), -\Psi\rangle_{\Omega}+O(\epsilon).
\end{align}
Here we denote $\langle \sqrt{-1}R^g(u,\o{u}), -\Psi\rangle_{\Omega}=(-\Psi)_{\alpha\b{\delta}}g^{\alpha\b{\beta}}g^{\gamma\b{\delta}}R^g_{\gamma\b{\beta}\sigma\b{\tau}}u^\sigma\b{u}^\tau$.

Substituting (\ref{2.9}), (\ref{2.8}) and  (\ref{2.7}) into (\ref{3.12}), we obtain
\begin{align}\label{3.24}
	\frac{\p}{\p t}\left(\p\b{\p}\log G(U,\o{U})\right)=\Delta^H_{\Omega}(\p\b{\p}\log G(U,\o{U}))+\langle \sqrt{-1}R^g(u,\o{u}), -\Psi\rangle_{\Omega}-\left|i_u\p^V\Psi\right|^2_{\Omega}.
\end{align}at the point $(z_0, [v_0], t_0)$, which completes the proof of Proposition \ref{Mainprop}.

Now we define a horizontal and real $(1,1)$-form $T$ as follow,
\begin{align}\label{defnT}
(-\sqrt{-1})T(X,\o{X}):=\langle \sqrt{-1}R^g(X,\o{X}), -\Psi\rangle_{\Omega}-\left|i_X\p^V\Psi\right|^2_{\Omega}	
\end{align}
for any horizontal vector $X=X^\alpha\frac{\delta}{\delta z^{\alpha}}$.
And we assume that  $T$ satisfies the null eigenvector assumption (see Theorem \ref{thm1}), by Theorem \ref{thm1}, we obtain
\begin{thm}\label{thm2}
Let $\pi: (E^*, G_0)\to M$ be a  holomorphic Finsler vector bundle over $M$ with $\sqrt{-1}\p\b{\p}\log G_0\geq 0$. Consider the following flow over the projective bundle $p:P(E^*)\to M$: 
\begin{align}\label{flow2}
\begin{cases}
&\frac{\p}{\p t}\log G=\Delta_{\Omega}\log G,\\
& \omega_{FS}(G)>0,\\
&G(0)=G_0,	
\end{cases}
\end{align}
where $\Omega=\omega(G)+\omega_{FS}(G)$, $\omega(G)=p^*\omega$, $\omega$ is a K\"ahler metric on $M$  depending on the Finsler metric $G$. 
If the horizontal $(1,1)$-form $T$ satisfies the null eigenvector assumption, then 
\begin{align*}
\sqrt{-1}\p\b{\p}\log G(t)\geq 0
\end{align*}
for all $t\geq 0$ such that the solution exists.
\end{thm}

\section{Applications}\label{sec3}
In this section, we will give two applications of Theorem \ref{thm2}.

\subsection{The case of curve}

In this subsection, we consider the case of $\dim M=1$, i.e. $M$ is a curve. In this case, any Hermitian metric $$\omega=\sqrt{-1}gdz\wedge d\b{z}$$ on $M$ is K\"ahler automatically. The Gaussian curvature is then given by 
\begin{align}
K=-\frac{1}{g}\frac{\p^2}{\p z\p\b{z}}\log g=\frac{1}{g^2}R^g(\frac{\p}{\p z},\frac{\p}{\p \b{z}},\frac{\p}{\p z},\frac{\p}{\p \b{z}})=:\frac{1}{g^2}R^g_{z\b{z}z\b{z}}.	
\end{align}
Now we assume that 
\begin{align}
\Omega=p^*\omega+\omega_{FS}(G),	
\end{align}
where $\omega=\omega(G)$ is a metric on $M$ depending smoothly on the Finsler metric $G$. Then at the point $(z_0,[v_0], t_0)$, by (\ref{3.111}), one has
\begin{align}
\begin{split}
(-\sqrt{-1})T(u,\o{u})&=\langle \sqrt{-1}R^g(u,\o{u}), -\Psi\rangle_{\Omega}-\left|i_u\p^V\Psi\right|^2_{\Omega}	\\
&=(-\Psi)_{z\b{z}}g^{-2}R^g_{z\b{z}z\b{z}}|u|^2-f^{a\b{b}}\p_a\Psi_{z\b{z}}\o{\p_b\Psi_{z\b{z}}}|u|^2g^{-1}\\
&=K\cdot\sqrt{-1}\Psi(u,\o{u})-f^{a\b{b}}\left(\p_a(-\sqrt{-1}\Psi(U,\o{U}))+\sqrt{-1}\Psi(\p_a U, u)\right)\o{\p_b\Psi_{z\b{z}}}g^{-1}\\
&=0,
\end{split}
\end{align}
since $i_u\Psi=0$ and $\sqrt{-1}\Psi(U,\o{U})$ attains its local minimal value at the point $(z_0,[v_0], t_0)$. Therefore, we prove
\begin{prop}\label{prop3}
If $M$ is a curve, then the semipositivity of the curvature of $\mc{O}_{P(E^*)}(1)$ is preserved along the flow (\ref{flow2}).	
\end{prop}

In particular, if $G_0=h_0^{i\b{j}}v_i\b{v}_j$ comes from a Hermitian metric $(h_0^{i\b{j}})$ of $E^*$ and 
\begin{align}
\Omega=p^*\omega+\omega_{FS}(G)	
\end{align}
for a fixed Hermitian metric $
\omega$, by Remark \ref{rem2} (1), (\ref{flow2}) is equivalent to the following Hermitian-Yang-Mills flow:
\begin{align}\label{HYM}
\begin{cases}
	&h^{-1}\cdot \frac{\p h}{\p t}+\Lambda R^h+(r-1)I=0\\
	&(h_{i\b{j}}(t))>0,\\
	&h_{i\b{j}}(0)=(h_0)_{i\b{j}}.
	\end{cases}
\end{align}
 By Proposition \ref{prop2} and Proposition \ref{prop3}, we have
\begin{cor}\label{cor3.2}
	If $M$ is a curve, then the Griffiths semipositivity  is preserved along the Hermitian-Yang-Mills flow (\ref{HYM}).
\end{cor}

\subsection{K\"ahler-Ricci flow}

In this section, we assume that $E=TM$. As the discussion in Remark \ref{rem2} (2), if we take
\begin{align*}
\omega(G)=\sqrt{-1}g_{\alpha\b{\beta}}dz^{\alpha}\wedge d\b{z}^{\beta},
\end{align*}
where $(g_{\alpha\b{\beta}})$ denotes the inverse of the matrix  $\left(\frac{\p^2 G}{\p v_{\alpha}\p\b{v}_{\beta}}\right)$. And $G_0=g_0^{\alpha\b{\beta}}v_\alpha\b{v}_\beta$ is a strongly pseudoconvex complex Finsler metric on $T^*M$ induced by the following K\"ahler metric 
\begin{align*}
\omega_0=\sqrt{-1}(g_0)_{\alpha\b{\beta}}dz^{\alpha}\wedge d\b{z}^{\beta}.
\end{align*}
By (\ref{KR}), the flow (\ref{flow2}) is equivalent to the following K\"ahler-Ricci flow
\begin{align}\label{KR1}
\begin{cases}
	& \frac{\p\omega}{\p t}+\text{Ric}(\omega)+(n-1)\omega=0,\\
	&\omega>0,\\
	&\omega(0)=\omega_0.
\end{cases}	
\end{align}
The solution of (\ref{flow2}) is induced from the K\"ahler metric $\omega=\sqrt{-1}g_{\alpha\b{\beta}}dz^{\alpha}\wedge d\b{z}^{\beta}$. In this case, 
\begin{align}
\begin{split}
\langle \sqrt{-1}R^g(u,\o{u}), -\Psi\rangle_{\Omega}&=(-\Psi)_{\alpha\b{\delta}}g^{\alpha\b{\beta}}g^{\gamma\b{\delta}}R^g_{\gamma\b{\beta}\sigma\b{\tau}}u^\sigma\b{u}^\tau\\	
&=\frac{1}{G}R^g_{\mu\b{\nu}\alpha\b{\delta}}g^{\mu\b{\sigma}}g^{\b{\nu}\tau}\o{v_{\sigma}}v_{\tau}g^{\alpha\b{\beta}}g^{\gamma\b{\delta}}R^g_{\gamma\b{\beta}\sigma\b{\tau}}u^\sigma\b{u}^\tau\\
&=\sum_{\alpha,\beta=1}^{\dim M}R^g(V,\o{V}, e_\alpha, \o{e_\beta})R^g(e_\beta,\o{e_\alpha},u,\o{u}),
\end{split}
\end{align}
where $V=\frac{1}{\sqrt{G}}g^{\alpha\b{\sigma}}\o{v_{\sigma}}\frac{\p}{\p z^{\alpha}}$ and $\{e_\alpha\}$ is a local orthonormal basis of $(TM, \omega)$. On the other hand, by (\ref{3.30}), one has at the point $(z_0,[v_0], t_0)$,
\begin{align}
\begin{split}
\left|i_u\p^V\Psi\right|^2_{\Omega}&=f^{a\b{b}}\p_a\Psi_{\alpha\b{\beta}}\o{\p_b\Psi_{\gamma\b{\tau}}}u^\alpha\o{u}^{\gamma}g^{\tau\b{\beta}}	\\
&=\frac{1}{G}(R^g)^{\mu\b{\sigma}}_{~~\alpha\b{\beta}}\o{v_{\sigma}}u^{\alpha} (R^{g})^{\delta\b{\nu}}_{~~\tau\b{\gamma}}v_{\delta}\o{u}^{\gamma}g_{\mu\b{\nu}}g^{\tau\b{\beta}}\\
&=\sum_{\alpha,\beta=1}^{\dim M}\left|R^g(V,\o{e_{\alpha}},u, \o{e_{\beta}})\right|^2.
\end{split}
\end{align}
Therefore, 
\begin{align}
	(-\sqrt{-1})T(u,\o{u})=\sum_{\alpha,\beta=1}^{\dim M}\left(R^g(V,\o{V}, e_\alpha, \o{e_\beta})R^g(e_\beta,\o{e_\alpha},u,\o{u})-\left|R^g(V,\o{e_{\alpha}},u, \o{e_{\beta}})\right|^2\right)\geq 0
\end{align}
by \cite[Page 254, Claim 2.2]{Bou}.
From Proposition \ref{prop2} and Theorem \ref{thm2}, we can reprove the following Mok's proposition, which is contained in \cite[Proposition 1.1]{Mok} (see also \cite[Theorem 5.2.10]{Bou}).
\begin{prop}[{\cite[Proposition 1.1]{Mok}}]
	If $(M, \omega_0)$ is a compact K\"ahler manifold with nonnegative holomorphic bisectional curvature, then the nonnegativity is preserved along the K\"ahler-Ricci flow (\ref{KR1}).
\end{prop}


\begin{thebibliography}{99}

\bibitem{AP} M. Abate, G. Patrizio, Finsler Metrics- A Global Approach, LNM 1591, Springer-Verlag, Berlin Heidelberg, 1994.

\bibitem{Aikou} T. Aikou, {\it Finsler Geometry on complex vector bundles}, Riemann-Finsler Geometry, MSRI Pulblications {\bf 50}, (2004), 83-105.

\bibitem{Atiyah} M. Atiyah, R. Bott, {\it The Yang-Mills equations over Riemann surfaces}, Phil. Trans. Roy. Soc. London A {\bf 308} (1982), 524-615.

\bibitem{Bando} S. Bando, {\it On three-dimensional compact K\"ahler manifolds of nonnegative bisectional curvature}, J. Differential Geometry {\bf 19} (1984), 283-297.


\bibitem{Bou} S. Boucksom, P. Eyssidieux, V. Guedj, An Introduction to the K{\"a}hler-Ricci Flow, {\bf 2086} (2013), Springer.


\bibitem{Cao} H.-D. Cao, {\it Deformation of K\"ahler metrics to K\"ahler-Einstein metrics on compact K\"ahler manifolds}, Invent. Math. {\bf 81} (1985), No. 2, 359-372.

\bibitem{Cao-Wong} J. Cao, P.-M. Wong, {\it Finsler geometry of projectivized vector bundles}, J. Math. Kyoto Univ. {\bf 43} (2003), No.2, 396-410.

\bibitem{Chen} X. Chen, {\it On K\"{a}hler manifolds with positive orthogonal bisectional curvature}, Advance in Mathematics {\bf 215} (2007), 427-445.

\bibitem{Chen1} X. Chen, S. Sun, G. Tian, {\it A note on K\"ahler-Ricci soliton,} Int. Math. Res. Not. IMRN 2009, no. {\bf 17}, 3328-3336.

\bibitem{Chow} B. Chow, D. Knopf, The Ricci flow: an introduction, Mathematical
Surveys and Monographs, vol. {\bf 110}, American Mathematical Society, Providence, RI, 2004.



\bibitem{Don1}  S. K. Donaldson, {\it Anti-self-dual Yang-Mills connections over complex algebraic surfaces and stable vector bundles}, Proc. London Math. Soc. {\bf 50} (1985), 1-26.


\bibitem{FLW} H. Feng, K. Liu, X. Wan, {\it Chern forms of holomorphic Finsler vector bundles and some applications}, Inter. J. Math. {\bf 27} (2016), No. 4, 1650030.

\bibitem{FLW2} H. Feng, K. Liu, X. Wan, {\it A Donaldson type functional on a holomorphic Finsler vector bundle}, Math. Ann. {\bf 369} (2017), no. 3-4, 997-1019. 

\bibitem{FLW1} H. Feng, K. Liu, X. Wan, {\it Geodesic-Einstein metrics and nonlinear stabilities},  Trans. Amer. Math. Soc. {\bf 371} (2019), no. 11, 8029-8049.

\bibitem{Gill} M. Gill, {\it Convergence of the parabolic complex Monge-Amp\`ere equation on compact Hermitian manifolds,} Comm. Anal. Geom. {\bf 19} (2011), 277-303.

\bibitem{Gri} P. Griffiths, {\it Hermitian differential geometry, Chern classes and positive vector bundles,} Global Analysis, papers in honor of K. Kodaira, Princeton Univ. Press, Princeton (1969), 181-251.

\bibitem{Gu} H. Gu, {\it A new proof of Mok's generalized Frankel conjecture theorem}, Proc. Amer. Math. Soc. {\bf 137} (2009), no.  3, 1063-1068.

\bibitem{Gu1} H. Gu, Z. Zhang, {\it An extension of Mok's theorem on the generalized Frankel conjecture}, Science China Mathematics {\bf 53} (2010), no. 5, 1253-1264.

\bibitem{Har} R. Hartshorne, {\it Ample vector bundles}, Inst. Hautes Etudes, Sci. Publ. Math. No. {\bf 29} (1966), 63-94. 

\bibitem{Ko1} S. Kobayashi, {\it Negative vector bundles and complex Finsler structures}, Nagoya Math. J. Vol. {\bf 57} (1975), 153-166.

\bibitem{Ko3} S. Kobayashi, Differential Geometry of Complex Vector Bundles, Iwanami-Princeton Univ. Press, 1987.

 \bibitem{Liu2} K. Liu, X. Yang, {\it Geometry of Hermitian manifolds}, International Journal of Mathematics {\bf 23} (2012), No. 6, 1250055 1-40. 

 \bibitem{Liu1} K. Liu, X. Sun, X. Yang, {\it Positivity and vanishing theorems for ample vector bundles},  J. Algebraic Geom. {\bf 22} (2013), No. 2, 303-331.
 
 \bibitem{Mok} N. Mok, {\it The uniformization theorem for compact K\"ahler manifolds of nonnegative holomorphic bisectional curvature}, J. Differ. Geom. {\bf 27} (1988), No. 2, 179-214.
 
 \bibitem{Mori} S. Mori, {\it Projective manifolds with ample tangent bundles}, Ann. of Math. (2) {\bf 110} (1979), no. 3, 593-606.
 
  \bibitem{Munteanu} G. Munteanu, Complex Spaces in Finsler, Lagrange and Hamilton Geometries, Kluwer Academic Publishers, 2004.
 

  
 \bibitem{Niu} Y. Niu, {\it Maximum principles for real $(p, p)$-forms on K\"ahler manifolds}, Geom Dedicata  {\bf 149} (2010), 363-371, DOI 10.1007/s10711-010-9486-7.
 
 \bibitem{Phong} D. H. Phong, J. Song, J. Sturm, B. Weinkove,
 {\it The K\"ahler-Ricci flow with positive bisectional curvature}, Invent. Math. {\bf 173} (2008), no. 3, 651-665.
 
 \bibitem{Siu} Y.-T. Siu, S.-T. Yau, {\it Compact K\"{a}hler manifolds of positive bisectional curvature}, Invent. Math. {\bf 59} (1980), 189-204.
 
 \bibitem{Streets} J. Streets, G. Tian, {\it Hermitian curvature flow}, Journal of the European Mathematical Society, {\bf 13} (2011),  601-634.
 
 \bibitem{Streets1} J. Streets, G. Tian, {\it A parabolic flow of pluriclosed metrics}, Int. Math. Res. Not. IMRN 2010, no. {\bf 16}, 3101-3133.
 
 \bibitem{Streets2} J. Streets, G. Tian, {\it Regularity results for pluriclosed flow}, Geom. Topol. {\bf 17} (2013), no. 4, 2389-2429.
 
 \bibitem{To1} V. Tosatti, B. Weinkove, {\it On the evolution of a Hermitian metric by its Chern-Ricci form}, J Differential Geom. {\bf 99} (2015), 125-163.
 
 \bibitem{To2}  V. Tosatti, B. Weinkove, X. Yang, {\it Collapsing of the Chern-Ricci flow on elliptic surfaces}, Math Ann, {\bf 362} (2015), 1223-1271.
 
 \bibitem{Yury} Y. Ustinovskiy, {\it The Hermitian curvature flow on manifolds with non-negative Griffiths curvature}, arXiv: 1604.04813v2, 2016.

 
 \bibitem{Wan} X. Wan, {\it Holomorphic sectional curvature of complex Finsler manifolds}, The Journal of Geometric Analysis,  J. Geom. Anal. {\bf 29} (2019), no. 1, 194-216.
 
 \bibitem{Wan1} X. Wan, G. Zhang, {\it The asymptotic of curvature of direct image bundle associated with higher powers of a relative ample line bundles}, arXiv: 1712.05922v1, 2017. 
 
 \bibitem{Yang} X. Yang, {\it The Chern-Ricci flow and holomorphic bisectional curvature,}. Sci. China Math. {\bf 59} (2016), no. 11, 2199-2204.
 
 

\end{thebibliography}
\end{document}